\documentclass{article}
\usepackage[utf8]{inputenc}
\usepackage[affil-it]{authblk}
\usepackage{graphicx} 
\usepackage{subcaption}
\usepackage{ulem}
\usepackage{comment,amsmath,amsbsy,amssymb,graphicx,dsfont,upgreek,textcomp,braket,setspace,blindtext,hyperref,verbatim,amsthm,mathrsfs,mathtools,graphicx,float,enumerate,cleveref,cases,bigints}
\usepackage[font={scriptsize,it}]{caption}
\usepackage[margin=1in]{geometry}
\hypersetup{
    colorlinks=true, 
    linktoc=all,     
    linkcolor=blue,  
}
\usepackage{tocloft} 
\usepackage[table,svgnames]{xcolor}
\usepackage{tikz}
\usetikzlibrary{arrows}
\usepackage[toc,page]{appendix}
\usepackage{xcolor}

\usepackage[font={normalsize,it}]{caption}

\Crefname{appsec}{appendix}{appendices}
\numberwithin{equation}{section}

\newtheorem{theorem}{Theorem}[section]
\newtheorem{lemma}{Lemma}[section]
\newtheorem{corollary}{Corollary}[section]
\newtheorem{remark}{Remark}[section]

\newtheorem{proposition}{Proposition}[section]

\title{Localization and global dynamics in the long-range discrete nonlinear Schr\"odinger equation}\vspace{-1ex}
\author{Brian Choi\thanks{Corresponding author. University of Tennessee at Chattanooga, \texttt{choigh@bu.edu}}, Austin Marstaller\thanks{Southern Methodist University, \texttt{amarstaller@smu.edu}}, Alejandro Aceves\thanks{Southern Methodist University, \texttt{aaceves@smu.edu}}%
\vspace{-2ex}  }

\date{\today\vspace{-1ex}}
\begin{document}
\maketitle\vspace{-4ex}
\maketitle
\begin{abstract}

We study localization, pinning, and mobility in the fractional discrete nonlinear Schrödinger equation (fDNLS) with generalized power-law coupling. A finite-dimensional spatial-dynamics reduction of the nonlocal recurrence yields onsite and offsite stationary profiles; their asymptotic validity, orbital stability of onsite solutions, and $l^2$ proximity to the exact lattice solutions are established. Using the explicit construction of localized states, it is shown that the spatial tail behavior is algebraic for all $\alpha > 0$. The Peierls–Nabarro barrier (PNB) is computed, and the parameter regimes are identified where it nearly vanishes; complementary numerical simulations explore mobility/pinning across parameters and exhibit scenarios consistent with near-vanishing PNB. We also analyze modulational instability of plane waves, locate instability thresholds, and discuss the role of nonlocality in initiating localization. Finally, we establish small-data scattering, and quantify how fDNLS dynamics approximates the nearest-neighbor DNLS on bounded times while exhibiting distinct global behavior for any large $\alpha$.
\end{abstract}



\section{Introduction}
Studies in arrays of coupled nonlinear oscillators remains a field of intense research whether it aims at understanding thermalization or the emergence of localized coherent structures, or global synchronization as exemplified by the FPUT or the Kuramoto oscillator model, respectively. Their relevance goes beyond theory as they model systems in a wide range of applications including photonics, plasmonics \cite{eisenberg1998discrete,aceves1996discrete}, Bose-Einstein condensates \cite{anderson1998macroscopic,trombettoni2001discrete}, and biological/chemical phenomena \cite{Mingaleev1999,PhysRevE.55.6141}. For more comprehensive surveys, see \cite{kevrekidis,SulemSulem:NLS_book}. In all these cases, the emerging models are large systems of coupled ordinary differential equations. Similarly for continuous fields, universal equations such as the nonlinear Schr\"odinger equation (NLSE) or the sine-Gordon equation have provided a platform to study many features in nonlinear wave phenomena including the existence and interactions of solitons, and the formation of coherent structures and singular blow-up dynamics.

Specific to nonlinear photonics, two canonical models are the NLSE and the discrete nonlinear Schrödinger equation (DNLS) given by
\begin{equation}\label{dnls}
    i \dot{u}_n = -\epsilon \delta^2 u_n - |u_n|^2 u_n,\ (n,t) \in \mathbb{Z} \times \mathbb{R},\ \epsilon>0,
\end{equation}
where $\delta^2 u_n := \sum\limits_{|j-n|=1} u_j - 2u_n$.
The discrete model describes waveguide/resonator arrays with nearest-neighbor coupling, whereas the continuum model captures, e.g., pulse propagation in optical fibers or waveguides. The present work focuses on a nonlocal discrete variant with long-range interactions (LRI), where the coupling strength decays algebraically. The baseline model is
\begin{equation}\label{fdnls}
i \dot{u}_n = \epsilon \sum_{m\in\mathbb{Z} \setminus \{n\} } \frac{u_n - u_m}{|n - m|^{1+\alpha}} - |u_n|^2 u_n,\ (n,t) \in \mathbb{Z} \times \mathbb{R},\ \epsilon>0,\ \alpha > 0,
\end{equation}
referred to as the fractional discrete NLS (fDNLS) for convenience. We note that for $\alpha \in (0,2]$, \eqref{fdnls} can be viewed as a discretization of the continuum fractional NLSE \eqref{continuum_fnls}; instead, the analysis here treats more general LRI profiles \eqref{generalmodel} which extends to all $\alpha >0$. Recent similar studies include nonlocal DNLS dynamics and higher-dimensional extensions \cite{PhysRevE.55.6141,molina2020two,johansson1998switching,rasmussen1998localized}; for example, localized modes on a square lattice were reported numerically in \cite{molina2020two}. The results below extend DNLS to LRI coupling and complement the literature by (i) providing a rigorous asymptotic characterization of spatial decay of localized states across the full range of $\alpha$, (ii) analyzing modulational instability and the emergence of coherent structures, and (iii) deriving the Peierls–Nabarro barrier (PNB) and indicating its role on mobility.

Unlike the multitude of experimental results on waveguide arrays modeled by the DNLS, at this time there is no concrete photonics-based array for which the fDNLS is an experimentally-verifiable model. We do believe, however, the results could pave the way to future realizations where nonlocality presents a new degree of freedom. As it relates to the continuum analogue
\begin{equation}\label{continuum_fnls}
    i\partial_t U = (-\Delta)^{\frac{\alpha}{2}} U - |U|^2 U,\ (x,t) \in \mathbb{R}^{d+1},
\end{equation}
where $(-\Delta)^{\frac{\alpha}{2}}u(x) := c(d,\alpha) \text{p.v.} \int\limits_{\mathbb{R}^{d}} \frac{u(x) - u(y)}{|x-y|^{d+\alpha}}dy$, has been proposed as a model for an optical cavity \cite{longhi2015fractional} whereby a proper design of lenses allows one to engineer diffraction to behave as $|k|^{\alpha},\ 0 < \alpha < 2$ in the Fourier representation. The suggested benefit of such a design is to produce laser beam outputs with unique Airy-like profiles as opposed to the classical Gaussian-like outputs.

The transition between continuum and discrete descriptions, either discretizing a continuum model or taking a continuum limit of an intrinsically discrete one, is natural but often non-trivial. For \eqref{fdnls}, \cite{kirkpatrick2013continuum} identifies three distinct continuum limits depending on $\alpha$. In \cite{malomed2024}, a 1D version of \eqref{fdnls} with $\alpha \in (0,2]$ is used to develop a discretized model, complementary to the analysis here. That study aligns with \Cref{MI} in spirit but concentrates on $\alpha \in (0,2]$, whereas \Cref{MI} emphasizes larger $\alpha$. The core results in \Cref{solution_DNLS} are to be solely representative of the discrete model, for which the relevant comparison is between DNLS and fDNLS.

A central question for nonlinear lattices is the transport properties of localized energy: mobility along the array requires transitions between onsite and offsite configurations which encounters the Peierls–Nabarro energy barrier induced by discreteness. In the continuum power-type NLSE (e.g., \eqref{continuum_fnls} with $\alpha = 2$), the Galilean invariance generates traveling families by boosting ground states with a proper momentum "kick", and in the mass-critical case the pseudoconformal symmetry underlies excitation-threshold results for strong nonlinearities \cite{Weinstein1983NonlinearSE,weinstein1989nonlinear}. Such symmetries are absent on lattices, so a nonzero PNB is expected. The asymptotic analysis here confirms a nonlocal PNB, and numerical simulations across $\alpha$ values show that initially boosted localized states generically pin to lattice sites.

One of the objectives of our work here, is to characterize modulational instability (MI) in fDNLS as a function of $\alpha$: MI thresholds and unstable $k$-bands are identified and the dependence on $\alpha$ is discussed (including recovering classical DNLS as $\alpha \rightarrow \infty$). Qualitative post-MI outcomes are outlined as a reminder of the standard pathway to coherent structures in dispersive media \cite{ZAKHAROV2009540,dauxois1993energy,daumont1997modulational}. Prior lattice studies include nonlocal hopping and interactions in 1D \cite{gori2013}, dipole–dipole–induced dispersion in higher dimensions under long-range approximations \cite{christiansen1998solitary}, and hybrids with continuum dispersion and two discrete directions coupled by inverse powers \cite{hadvzievski2003}, with emergence of coherent structures reported in related one-discrete models \cite{Copeland:20}. Rather than fixing a single nonlocal model, the analysis treats $\alpha$ as the primary control parameter, clarifying how nonlocality and discreteness jointly set the MI band and shape the post-instability states.

While the focus is the discrete model, its continuum limit is manifested via the scaling $\epsilon = c(d,\alpha) h^{-\alpha}$ where $\alpha \in (0,2]$ and $h>0$ the lattice spacing. The continuum theory is more delicate and requires specialized well-posedness tools in Sobolev spaces, including the Littlewood–Paley theory \cite{DINH2019117,1534-0392_2015_6_2265,dinh:hal-01426761,choi2022well}. The continuum limit is nontrivial: the long-range coupling can be recast via fractional-derivative operators \cite{tarasov2006continuous}, and strong discrete-to-continuum convergence has been proved using uniform discrete Strichartz estimates \cite{hong2019strong}, with two-dimensional fractional extensions \cite{choi2023continuum,christiansen1998solitary}. In this limit, the Galilean invariance is recovered and PNB vanishes (indeed exponentially for DNLS and long-range variants) \cite{JenWeinLocal,jenkinson_weinstein_2017}. By contrast, the results here address the anti-continuum regime, where new PNB formulas and asymptotics are derived on the discrete lattice.

This paper is organized as follows. In \Cref{background}, the main model and its properties are introduced. In \Cref{wellposedness}, the global dynamics corresponding to $\alpha \gg 1$ is investigated. As $\alpha \rightarrow \infty$, \eqref{fdnls} converges, locally in time but not globally, to the DNLS. In \Cref{MI}, linear stability analysis on the CW solution of fDNLS is shown with explicit regions of MI. The onset of nonlinear bound states resulting from MI is shown via numerical simulations. In \Cref{solution_DNLS}, the family of onsite and offsite solutions and the corresponding PNB are constructed asymptotically. Although the rigorous aspects of numerical analysis (consistency, convergence, etc) are not our main focus, we provide the numerical discretization of the periodic fractional Laplacian in \Cref{appendix}. 

\section{Background}\label{background}

In this paper, we consider the generalized nonlocal model (also considered in \cite{kirkpatrick2013continuum,tarasov2006continuous,jenkinson_weinstein_2017}) where the infinitesimal generator is given by
\begin{equation}\label{inf}
    (\mathscr{L}_\alpha f)_n = \sum_{m \neq n} J_{|n-m|}(f_n - f_m),\ (f_n) \in l^2(\mathbb{Z}),\ \alpha >0,
\end{equation}
where $J_n \geq 0$ is the $\alpha$-kernel satisfying the limit property
\begin{equation*}
    \lim\limits_{n \rightarrow \infty} n^{1+\alpha} J_n = A_\alpha \in (0,\infty),
\end{equation*}
and when $\alpha = \infty$, define it as an $\infty$-kernel if $\lim\limits_{n \rightarrow \infty} n^{1+\alpha} J_n = 0$ for all $\alpha >0$; assume that $(J_n)_{n=1}^\infty$ is not identically zero. Note that $\mathscr{L}_\alpha$ defines a family of self-adjoint, bounded linear operators on $l^2(\mathbb{Z})$. With the convention $\mathcal{F}[f](k) = \sum\limits_{n \in \mathbb{Z}} f_n e^{in  k},\ k \in \mathbb{T} = (-\pi,\pi]$, the Fourier symbol of $\mathscr{L}_{\alpha}$ is given by $2 \sum\limits_{n \neq 0} J_{|n|} \sin^2 \left(\frac{nk}{2}\right)$, and therefore $\| \mathscr{L}_{\alpha} \|_{l^2 \rightarrow l^2} \leq 4 J$ where $J = \sum\limits_{n=1}^{\infty} J_n$. In numerical applications, our focus lies in specific LRI kernels defined by $J_n = |n|^{-(1+\alpha)}$, but assume the general form of nonlocal interaction kernel unless otherwise specified.

A particular nonlinear model generated by \eqref{inf} is given by
\begin{equation}\label{generalmodel}
    i \dot{u}_n = \epsilon\mathscr{L}_\alpha u_n - |u_n|^2 u_n,\ u(0) = f \in l^2(\mathbb{Z}),\ \epsilon > 0,
\end{equation}
and the stationary model, by taking the ansatz $u_n(t) = e^{i\omega t} Q_n$ where $\omega >0,\ Q_n \in \mathbb{R}$, is given by 
\begin{equation}\label{generalmodel_stationary}
    -{\omega} Q_n = \epsilon\mathscr{L}_\alpha Q_n - Q_n^3.
\end{equation}
Under the flow \eqref{generalmodel}, there are at least two conserved quantities given by
\begin{equation}\label{energy}
N[u(t)] = \sum_{n} |u_n|^2,\ E[u(t)] = \frac{\epsilon}{2} \langle \mathscr{L}_\alpha u_n , u_n \rangle_{l^2} - \frac{1}{4} \sum_{n}|u_n|^4,    
\end{equation}
representing the particle number/mass/power and energy, respectively. The kinetic energy terms corresponding to DNLS \eqref{dnls} and fDNLS \eqref{fdnls} are given by $\frac{\epsilon}{2}\sum\limits_{n \in \mathbb{Z}} |u_{n+1} - u_n|^2$ and $\frac{\epsilon}{4}\sum\limits_{n,m: n\neq m} \frac{|u_n - u_{m}|^2}{|n-m|^{1 + \alpha}}$, respectively.

There are varying definitions of the fractional Laplacian on bounded domains, and hence a particular numerical discretization of $(-\Delta)^{\frac{\alpha}{2}}$ needs to be defined since numerical simulations depend on an appropriate spatial truncation. In \Cref{MI} where the evolution of $u_n(0) = A > 0,\ -N \leq n \leq N$ is studied, the periodic boundary condition is imposed whereas the zero exterior Dirichlet boundary condition is imposed to simulate localized wave solutions in \Cref{solution_DNLS}. The spectral definition of $(-\Delta)^{\frac{\alpha}{2}}$ is defined as a multiplier $|k|^{\alpha}$.

\section{Global well-posedness and small data scattering}\label{wellposedness}

The approximation of DNLS via fDNLS for large $\alpha$ (see \Cref{gwp_convergence}) need not hold for the long-time dynamics, and the exponential bound \eqref{exp_growth} provides no control global in time. The long-time dynamics of an extended Hamiltonian system exhibits a rich structure featuring multi-breathers, transition into chaos, and small data scattering, just to name a few, where such variety of features arises from conservation laws in stark contrast to dissipative systems. \Cref{gwp_divergence} demonstrates small data scattering uniform in $\alpha$ for sufficiently high nonlinearity. As a corollary, this proves the existence of strictly positive excitation thresholds of ground state solutions for fDNLS. Note that \cite{weinstein1999excitation} showed that $p \geq 5$ for DNLS (the power nonlinearity; see \eqref{gwp_divergence2}) is the sufficient and necessary condition for positive excitation thresholds. We leave it as a future work to investigate the case $5 \leq p < 7$ for fDNLS.

By the contraction mapping argument and the embedding $l^p(\mathbb{Z}) \hookrightarrow l^q(\mathbb{Z})$ whenever $p \leq q$, the well-posedness of \eqref{fdnls}, \eqref{dnls} is established. As long as the LRI is described by a self-adjoint operator and the nonlinear interaction, by a local nonlinearity, the following well-posedness result follows similarly. 

\begin{proposition}\label{gwp}
    Let $L$ be a bounded, self-adjoint operator on $l^2(\mathbb{Z})$ and $\Tilde{N}:\mathbb{C} \rightarrow \mathbb{C}$ such that
    \begin{equation*}
    \Tilde{N}(0) = 0,\ |\Tilde{N}(z_1) - \Tilde{N}(z_2)| \leq C\left(\max(|z_1|,|z_2|)\right)|z_1 - z_2|,    
    \end{equation*}
    where $C:[0,\infty) \rightarrow [0,\infty)$ is an increasing function. Then the initial-value problem
\begin{equation}\label{ivp}
    i\dot{u}_n = L u_n + \Tilde{N}(u_n),\ u(0) = f \in l^2(\mathbb{Z}),
\end{equation}
    is globally well-posed; for any $f \in l^2(\mathbb{Z})$, there exists a unique solution $u \in C^1_{loc}(\mathbb{R};l^2(\mathbb{Z}))$ to \eqref{ivp} such that $u(0) = f$ and the data-to-solution map $f \mapsto u$ is Lipschitz continuous.
\end{proposition}

\begin{proposition}\label{gwp_convergence}
Assume the hypotheses of \Cref{gwp} hold. Let $\{L_\alpha\}_{\alpha>0}$ and $L$ be bounded, self-adjoint operators on $l^2(\mathbb{Z})$ such that $L_\alpha \xrightarrow[\alpha \rightarrow \infty]{} L$ in norm. Let $u^{(\alpha)},v \in C^1_{loc}(\mathbb{R};l^2(\mathbb{Z}))$ be the well-posed solutions to \eqref{ivp} given by $L_\alpha, L$, respectively, satisfying $u^{(\alpha)}(0) = f^{(\alpha)},\ v(0) = g$, both in $l^2(\mathbb{Z})$. Assume $\sup\limits_{\alpha > 0}\| f^{(\alpha)} \|_{l^2} \leq M$. Then there exists $C = C(M, \| g \|_{l^2})>0$ such that for all $t \geq 0$,
\begin{equation}\label{exp_growth}
    \| u^{(\alpha)}(t) - v(t) \|_{l^2} \leq e^{Ct}\Big(\| f^{(\alpha)}-g \|_{l^2} + t \| L_\alpha - L\| \cdot \| g \|_{l^2}\Big).
\end{equation}
\end{proposition}

\begin{proof}
    A sketch of proof is given. For notational brevity, say $u = u^{(\alpha)},\ f = f^{(\alpha)}$. The well-posedness of $u,v$ follows by \Cref{gwp}. Setting $\phi = u - v$, we have
    \begin{equation*}
        i\dot{\phi}_n = L_\alpha \phi_n + (L_\alpha - L)v_n + \Tilde{N}(u_n) - \Tilde{N}(v_n).
    \end{equation*}
    By integrating, it follows that
    \begin{equation*}
        \phi_n(t) = e^{-it L_\alpha} \phi_n(0) - i \int_0^t e^{-i(t-t^\prime)L_\alpha} \bigl\{(L_\alpha - L)v_n(t^\prime) + \Tilde{N}(u_n(t^\prime)) - \Tilde{N}(v_n(t^\prime))\bigl\}dt^\prime, 
    \end{equation*}
    and by the triangle inequality, the unitarity of $e^{-it L_\alpha}$, the conservation of particle numbers, and the embedding $l^2(\mathbb{Z})\hookrightarrow l^\infty(\mathbb{Z})$, we have
    \begin{equation*}
    \| \phi(t) \|_{l^2} \leq \| \phi(0) \|_{l^2} + t \| L_\alpha - L \| \cdot \| g \|_{l^2}  + C(\max(M,\| g \|_{l^2})) \int_0^t \| \phi(t^\prime )\|_{l^2} dt^\prime,
    \end{equation*}
    where $C>0$ is the (local) Lipschitz constant of the nonlinearity $\Tilde{N}$. The proof follows from the Gronwall's inequality.   
\end{proof}

\begin{remark}\label{rmk2}
    In the context of LRI and DNLS, the operators defined by 
\begin{equation*}
    (L_\alpha f)_n = \epsilon\sum\limits_{m \neq n} \frac{f_n - f_m}{|n-m|^{1+\alpha}},\ L = -\epsilon \delta^2,     
\end{equation*}
    on $l^2(\mathbb{Z})$ are Fourier multipliers with the symbols given by
\begin{equation}\label{discrete_symbol}
    \sigma_\alpha(k) = 2\epsilon \sum_{n \neq 0} \frac{\sin^2\left(\frac{n  k}{2}\right)}{|n|^{1+\alpha}};\ \sigma(k) = 2\epsilon \sum_{|n|=1}\sin^2\left(\frac{n k}{2}\right).    
\end{equation}
Then the norm-convergence hypothesis of \Cref{gwp_convergence} is satisfied since
\begin{equation*}
    |\sigma_\alpha(k) - \sigma(k)| = 2\epsilon \left|\sum_{|n| \geq 2} \frac{\sin^2\left(\frac{nk}{2}\right)}{|n|^{1+\alpha}}\right| \leq \sum_{|n| \geq 2} \frac{2\epsilon}{|n|^{1+\alpha}} \lesssim \epsilon \int_2^\infty \frac{dr}{r^{1+\alpha}} = \frac{\epsilon}{\alpha 2^\alpha} \xrightarrow[\alpha \to \infty]{} 0,
\end{equation*}
uniformly in $k$. Hence by \eqref{exp_growth}, the short-time dynamics of fDNLS is well-approximated by that of the DNLS for large $\alpha$ on $\mathbb{Z}$. 
\end{remark}

\begin{theorem}\label{gwp_divergence}
Let $\mathscr{L}_\infty := \mathscr{L} = -\delta^2$ and $\mathscr{L}_\alpha$ be defined by $J_n = |n|^{-(1+\alpha)}$ by \eqref{inf}. Define $U(t) = e^{-it\mathscr{L}}$ and $U^{(\alpha)}(t) = e^{-it\mathscr{L}_\alpha}$. By \Cref{gwp}, let $u^{(\alpha)}, v := u^{(\infty)} \in C^1_{loc}(\mathbb{R};l^2(\mathbb{Z}))$ be the well-posed solutions to
\begin{equation}\label{gwp_divergence2}
\begin{split}
i \dot{u}_n^{(\alpha)} &= \mathscr{L}_\alpha u_n^{(\alpha)} + \mu |u_n^{(\alpha)}|^{p-1} u_n^{(\alpha)}\\
i \dot{v}_n &= \mathscr{L} v_n + \mu |v_n|^{p-1} v_n,
\end{split}   
\end{equation}
where $u^{(\alpha)}(0)=v(0)=f \in l^2(\mathbb{Z})$ and $\mu = \pm 1,\ p \geq 7$. Then there exists $\Tilde{\alpha} > 0$ such that for all $\Tilde{\alpha} < \alpha < \infty$, the data-to-asymptotic-state map is well-defined as an isometric homeomorphism in a small neighborhood of $l^2(\mathbb{Z})$ uniformly in $\Tilde{\alpha} < \alpha \leq \infty$; more precisely, there exists $\delta_1(\Tilde{\alpha}) > 0$ such that whenever $\| f \|_{l^2} < \delta_1$, there exists unique $u_+^{(\alpha)} \in l^2(\mathbb{Z})$ such that
\begin{equation}\label{scattering}
    \| u^{(\alpha)}(t) - U^{(\alpha)}(t) u_+^{(\alpha)}\|_{l^2} \xrightarrow[t \rightarrow \infty]{} 0.
\end{equation}
Furthermore there exists $\delta_2(\Tilde{\alpha})>0$ such that whenever $0 < \| f \|_{l^{\frac{5}{4}}} < \delta_2$, we have
\begin{equation}\label{dispersive_est}
    \| u^{(\alpha)}(t) \|_{l^5} \leq C(\Tilde{\alpha}) (1+|t|)^{-\frac{1}{5}} \| f \|_{l^{\frac{5}{4}}},
\end{equation}
and
\begin{equation}\label{singular_limit}
    \inf_{\alpha > \Tilde{\alpha}}\sup_{t \in [0,\infty)}\| u^{(\alpha)}(t) - v(t) \|_{l^2} > 0.
\end{equation}
\end{theorem}

\begin{lemma}\label{lem}
    There exists $\alpha_0^* > 0$ such that for all $\alpha > \alpha_0^*$, there exists $t_\alpha>2^\alpha$ and
    \begin{equation*}
        \| (U^{(\alpha)}(t_\alpha) - U(t_\alpha))f \|_{l^2} \geq c \| \widehat{f} \|_{L^{1/3}},
    \end{equation*}
    for all $f \in l^2(\mathbb{Z}) \setminus \{0\}$ where $c>0$ is independent of $\alpha$ and $f$.
\end{lemma}
\begin{proof}
    By the Plancherel Theorem and the reverse H\"older inequality,
\begin{align}\label{int}
\| (U^{(\alpha)}(t/4) - U(t/4))f\|_{l^2}^2 &= \frac{1}{2\pi} \bigintsss_{-\pi}^{\pi} \left|\exp \left(-it \sum_{m=1}^\infty\frac{\sin^2\left(\frac{mk}{2}\right)}{m^{1+\alpha}}\right) - \exp\left(-it \sin^2\left(\frac{k}{2}\right)\right)\right|^2 \cdot|\widehat{f}(k)|^2 dk\nonumber\\
&\gtrsim \| \widehat{f} \|_{L^{1/3}}^2 \left(\bigintsss_{-\pi}^{\pi} \left|\exp \left(-it \sum_{m=2}^\infty\frac{\sin^2\left(\frac{mk}{2}\right)}{m^{1+\alpha}}\right) - 1\right|^{-\frac{2}{5}} dk\right)^{-5}\nonumber\\
&\simeq \| \widehat{f} \|_{L^{1/3}}^2 \left(\bigintsss_{0}^{\pi} \left(1-\cos\left(t \sum\limits_{m=2}^\infty\frac{\sin^2\left(\frac{mk}{2}\right)}{m^{1+\alpha}}\right)\right)^{-\frac{1}{5}} dk\right)^{-5},
\end{align}
since $0 < \| \widehat{f} \|_{L^{1/3}} \lesssim \| \widehat{f} \|_{L^2} < \infty$. For the convenience of notation, let
\begin{equation*}
    X(k) = t \sum\limits_{m=2}^\infty\frac{\sin^2\left(\frac{mk}{2}\right)}{m^{1+\alpha}}.
\end{equation*}
Then the integral above on $k \in [0,\pi]$ is split into
\begin{equation*}
    \int_0^{\pi} (1-\cos X)^{-\frac{1}{5}} dk = \int_{\{1-\cos X < c\}} (1-\cos X)^{-\frac{1}{5}} dk + \int_{\{1-\cos X \geq c\}} (1-\cos X)^{-\frac{1}{5}} dk =: I + II,
\end{equation*}
where $c>0$ is sufficiently small. Then $II$ is $O(1)$ depending only on $c$, and therefore it suffices to estimate $I$.

First consider $0 \leq X < X_0$ where $X_0$ is the smallest positive real such that $1-\cos X_0 = c$, and suppose $t = 2^{1+\alpha}(2\pi + X_0)$. Then
\begin{equation}\label{admissible}
    0 \leq \sin^2 k < \frac{X_0}{2\pi + X_0} - 2^{1+\alpha}\sum_{m \geq 3} \frac{\sin^2(\frac{mk}{2})}{m^{1+\alpha}},
\end{equation}
and $k \in [0,\pi]$ that satisfies \eqref{admissible} are near $k=0,\pi$ since the series term is negligible for all $\alpha>0$ sufficiently large due to the uniform estimate 
\begin{equation}\label{technical}
    2^{1+\alpha}\sum_{m \geq 3} \frac{\sin^2(\frac{mk}{2})}{m^{1+\alpha}} < 2\left(\frac{1}{3}+\frac{1}{\alpha}\right) \left(\frac{2}{3} \right)^{\alpha},
\end{equation}
that follows from majorizing the series into an appropriate integral. The arguments for the estimation near $k = 0,\pi$ are similar, and therefore assume $k$ is near $0$. Note that \eqref{admissible} is satisfied for $0 \leq k < k_0$ where $k_0$ is the smallest positive real root of
\begin{equation*}
    \sin^2(k_0) = \frac{X_0}{2\pi + X_0} - 2^{1+\alpha}\sum_{m \geq 3} \frac{\sin^2(\frac{mk_0}{2})}{m^{1+\alpha}}.
\end{equation*}
Then for all $\alpha>0$ sufficiently large depending on $X_0$, we have $k_0 < k_0^*$ where $k_0^*$ is the smallest positive real root of $\sin^2(k_0^*) = \frac{X_0}{2\pi + X_0}$. By the Taylor expansion,
\begin{equation*}
    k_0 < k_0^* \lesssim X_0^{\frac{1}{2}},
\end{equation*}
and
\begin{equation}\label{technical3}
\begin{split}
\int_{\{0 \leq X < X_0\}} (1-\cos X)^{-\frac{1}{5}} dk &\lesssim \int_{\{0 \leq X < X_0\}} X^{-\frac{2}{5}} dk \lesssim \int_{\{0 \leq X < X_0\}} (2\pi + X_0)^{-\frac{2}{5}} \sin^{-\frac{4}{5}}(k) dk\\
&\lesssim \int_{0}^{k_0} k^{-\frac{4}{5}} dk \lesssim X_0^{\frac{1}{10}},    
\end{split}    
\end{equation}
where the implicit constants depend only on $c,X_0$.

Now consider $|X-2\pi| < X_0$. Since the analysis for $0 < X - 2\pi < X_0$ and $-X_0 < X - 2\pi <0$ are similar, we focus on the former. Then
\begin{align*}\label{sin}
2\pi < 2^{1+\alpha}(2\pi + X_0) \left(\frac{\sin^2 k}{2^{1+\alpha}} + \sum_{m \geq 3} \frac{\sin^2(\frac{mk}{2})}{m^{1+\alpha}}\right) &< 2\pi + X_0\nonumber\\
\frac{2\pi}{2\pi + X_0} - 2^{1+\alpha}\sum_{m \geq 3} \frac{\sin^2(\frac{mk}{2})}{m^{1+\alpha}} &< \sin^2 k < 1- 2^{1+\alpha}\sum_{m \geq 3} \frac{\sin^2(\frac{mk}{2})}{m^{1+\alpha}} < 1.  
\end{align*}
Define $k_1(\alpha),k_1^* \in (0,\frac{\pi}{2})$ such that 
\begin{equation*}
\begin{split}
\sin^2 k_1 &= \frac{2\pi}{2\pi + X_0} - 2^{1+\alpha}\sum_{m \geq 3} \frac{\sin^2(\frac{mk}{2})}{m^{1+\alpha}}\\
\sin^2 k_1^* &= \frac{2\pi}{2\pi + X_0}.
\end{split}    
\end{equation*}
The Taylor expansion of $\cos X$ near $X = 2\pi$ yields
\begin{equation*}
\begin{split}
\int_{\{0<X-2\pi<X_0\}} (1-\cos X)^{-\frac{1}{5}}dk &\lesssim \int_{E} |X-2\pi|^{-\frac{2}{5}}dk\\
&= \int_{k_1}^{\frac{\pi}{2}} |X-2\pi|^{-\frac{2}{5}}dk + \int_{\frac{\pi}{2}}^{\pi - k_1} |X-2\pi|^{-\frac{2}{5}}dk =: III+ IV  
\end{split}
\end{equation*}
where 
\begin{equation*}
E = \{k \in [0,\pi]:\frac{2\pi}{2\pi + X_0} - 2^{1+\alpha}\sum\limits_{m \geq 3} \frac{\sin^2(\frac{mk}{2})}{m^{1+\alpha}} < \sin^2 k < 1\} = (k_1,\pi - k_1).    
\end{equation*}
The term $IV$ can be estimated similarly as $III$, and therefore the work for $III$ is shown. For sufficiently large $\alpha>0$, the Taylor expansion of $X(k)$ near $k = k_1$ yields
\begin{equation*}
\begin{split}
X(k) &= X(k_1) + X^\prime(k_1) (k-k_1) + O(|k-k_1|^2)\\
&= 2\pi + 2^\alpha (2\pi + X_0) \sum_{m \geq 2} \frac{\sin (mk_1)}{m^\alpha} (k-k_1) + O(|k-k_1|^2).    
\end{split}
\end{equation*}
By arguing as \eqref{technical}, the coefficient of $k-k_1$ is bounded above by a constant independent of $\alpha$. Conversely, since $k_1 < k_1^*$ and $k_1 \xrightarrow[\alpha \rightarrow \infty]{}k_1^*$, we have $2k_1 \in (2k_1^* - \delta,2k_1^*)$ for some $\delta>0$ and $\sin (2k_1) >0$. The lower bound is given by
\begin{equation*}
    2^\alpha \sum_{m \geq 2} \frac{\sin (mk_1)}{m^\alpha}  \geq \sin (2k_1) - \sum_{m \geq 3} \left(\frac{2}{m}\right)^\alpha > \sin (2k_1^*) - \sum_{m \geq 3} \left(\frac{2}{m}\right)^\alpha \gtrsim 1,
\end{equation*}
where the implicit constant is independent of $\alpha$. Hence the linear coefficient is bounded above and below by a positive constant independent of $\alpha$, and similarly, the quadratic coefficient is bounded above uniformly in $\alpha$. Then
\begin{equation}\label{technical2}
    III \lesssim \int_{k_1}^{\frac{\pi}{2}} |k-k_1|^{-\frac{2}{5}}dk = O(1).
\end{equation}
Hence by \eqref{technical3}, \eqref{technical2}, and $t = 2^{1+\alpha} (2\pi + X_0)$, there exists $M>0$ independent of $\alpha$ such that
\begin{equation*}
    \int_{\{1- \cos X <c\}} (1-\cos X)^{-\frac{1}{5}} dk \leq \int_{\{0 < X < X_0\}} (1-\cos X)^{-\frac{1}{5}} dk + \int_{\{|X-2\pi| < X_0\}} (1-\cos X)^{-\frac{1}{5}} dk \leq M,
\end{equation*}
and this shows the lower bound of \eqref{int}.
\end{proof}

\begin{proof}[Proof of 
\Cref{gwp_divergence}]
The proof strategy is as follows. As \cite{stefanov2005asymptotic} in their analysis of DNLS, derive linear dispersive estimates (Strichartz estimates) uniform in $\alpha$ and apply them to establish the small data scattering \eqref{scattering}. By the fixed point argument in the Strichartz space, derive \eqref{dispersive_est}. Lastly estimate $\| U^{(\alpha)}(t) - U(t)\|$ for $t,\alpha \gg 1$ to obtain \eqref{singular_limit}.

By the discrete Fourier transform, $U^{(\alpha)}(t) f_n = \left(K_{t,\alpha} \ast f\right)_n$ where
\begin{equation*}
K_{t,\alpha}(n) := \frac{1}{2\pi} \int_{-\pi}^{\pi} e^{-it\phi(k)}dk,\ \phi(k) = 4\sum_{m=1}^\infty \frac{\sin^2\left(\frac{mk}{2}\right)}{m^{1+\alpha}} + \frac{nk}{t}.
\end{equation*}
By the Dominated Convergence Theorem,
\begin{equation*}
    \phi^{\prime\prime}(k) \xrightarrow[\alpha \rightarrow \infty]{} 2\cos(k),\ \phi^{\prime\prime\prime}(k) \xrightarrow[\alpha \rightarrow \infty]{} -2\sin(k), 
\end{equation*}
and therefore $\max\limits_{k \in [-\pi,\pi]} \left(|\phi^{\prime\prime}(k)|,|\phi^{\prime\prime\prime}(k)|\right) \gtrsim 1$ for $\alpha$ sufficiently large. In fact, the bound holds as long as $\alpha \geq \alpha_0 > 3$ under which the term-by-term differentiation of $\phi^{\prime\prime\prime}$ holds uniformly. By the Van der Corput Lemma \cite[p.334]{stein1993harmonic}, this implies
\begin{equation}\label{endpoint_interp}
    \sup_{n \in \mathbb{Z}}|K_{t,\alpha}(n)| \lesssim_{\alpha_0} (1+|t|)^{-\frac{1}{3}},
\end{equation}
where the inhomogeneous bound follows from applying the triangle inequality to $K_{t,\alpha}$. As a corollary, we obtain
\begin{equation}\label{strichartz2}
    \| U^{(\alpha)}(t)f \|_{l^s} \lesssim (1+|t|)^{-\frac{s-2}{3s}} \| f \|_{l^{s^\prime}},
\end{equation}
by the complex interpolation of linear operators where $s \in [2,\infty]$ and $s^\prime := \frac{s}{s-1}$; the unitarity of $U^{(\alpha)}(t)$ and \eqref{endpoint_interp} provide the endpoint estimates at $s = 2,\infty$, respectively. Let $(q,r) \in [2,\infty]^2$ satisfy $\frac{3}{q} + \frac{1}{r} = \frac{1}{2}$, and say that such pair is DNLS-admissible. A further corollary \cite[Theorem 1.2]{keel1998endpoint} implies
\begin{equation}\label{strichartz}
\begin{split}
\| U^{(\alpha)}(t) f \|_{L^q_t l^r(I\times \mathbb{Z})} &\lesssim \| f \|_{l^2},\\
\left|\left| \int_0^\infty U^{(\alpha)}(-t^\prime) N(t^\prime)dt^\prime \right|\right|_{l^2} &\lesssim \| N \|_{L^{\Tilde{q}^\prime}_t l^{\Tilde{r}^\prime}(\mathbb{R}\times \mathbb{Z})},\\
\left|\left| \int_0^t U^{(\alpha)}(t-t^\prime) N(t^\prime)dt^\prime \right|\right|_{L^q_t l^r(I\times \mathbb{Z})} &\lesssim \| N \|_{L^{\Tilde{q}^\prime}_t l^{\Tilde{r}^\prime}(I\times \mathbb{Z})},
\end{split}
\end{equation}
where $(\Tilde{q},\Tilde{r})$ is DNLS-admissible and $I \subseteq \mathbb{R}$. Note that the implicit constants depend on the DNLS-admissible pairs and $\alpha_0$.

The solution in the integral form satisfies
\begin{equation}\label{duhamel}
    u^{(\alpha)}(t) = U^{(\alpha)}(t)f - i\mu \int_0^t U^{(\alpha)}(t-t^\prime)\left(|u^{(\alpha)}(t^\prime)|^{p-1}u^{(\alpha)}(t^\prime)\right)dt^\prime.
\end{equation}
Let $X = C(\mathbb{R};l^2(\mathbb{Z})) \cap L^{\frac{6p}{5}}_t l^{\frac{2p}{p-5}}(\mathbb{R}\times\mathbb{Z})$ and let $\Gamma u$ be the RHS of \eqref{duhamel}. By the fixed point argument, $\Gamma$ has a unique fixed point in $X$ if $\| f \|_{l^2} = O(1)$ where the bound depends only on $\alpha_0$ since the Strichartz constant of \eqref{strichartz2} can be chosen uniformly. Then
\begin{equation*}\label{scattering2}
\begin{split}
\left|\left|\int_{t_1}^{t_2} U^{(\alpha)}(-t^\prime)\left(|u^{(\alpha)}(t^\prime)|^{p-1}u^{(\alpha)}(t^\prime)\right)dt^\prime \right|\right|_{l^2} \lesssim \| u^{(\alpha)}\|_{L^{\frac{6p}{5}}_t l^{p}([t_1,t_2]\times\mathbb{Z})}^p\leq \| u^{(\alpha)}\|_{L^{\frac{6p}{5}}_t l^{\frac{2p}{p-5}}([t_1,t_2]\times\mathbb{Z})}^p \xrightarrow[t_1,t_2 \rightarrow \infty]{} 0, 
\end{split}
\end{equation*}
where the first inequality is by \eqref{strichartz}, the second inequality by the H\"older's inequality and $p \geq 7$, and the last limit by $\| u^{(\alpha)}\|_{L^{\frac{6p}{5}}_t l^{\frac{2p}{p-5}}([t_1,t_2]\times\mathbb{Z})} \leq \| u^{(\alpha)}\|_{L^{\frac{6p}{5}}_t l^{\frac{2p}{p-5}}(\mathbb{R}\times\mathbb{Z})} \lesssim \| f \|_{l^2} < \infty$ with $t_1 \leq t_2$ without loss of generality. This shows the existence of $u_+^{(\alpha)}  = \lim\limits_{t\rightarrow\infty} U^{(\alpha)}(-t)u^{(\alpha)}(t) \in l^2(\mathbb{Z})$ and by \eqref{duhamel},
\begin{equation}\label{duhamel2}
    u^{(\alpha)}(t) = U^{(\alpha)}(t) u^{(\alpha)}_+ + i\mu \int_t^{\infty} U^{(\alpha)}(t-t^\prime)\left(|u^{(\alpha)}(t^\prime)|^{p-1}u^{(\alpha)}(t^\prime)\right)dt^\prime. 
\end{equation}

To show that the inverse of data-to-asymptotic-state (wave operator) is continuous, we argue by the fixed point theorem on $X_T = C([T,\infty);l^2(\mathbb{Z})) \cap L^{\frac{6p}{5}}_t l^{\frac{2p}{p-5}}([T,\infty)\times\mathbb{Z})$ where $\Gamma^\prime u$ is defined as the RHS of \eqref{duhamel2} and $T>0$ is chosen such that $\| U^{(\alpha)}(t) u_+^{(\alpha)}\|_{L^{\frac{6p}{5}}_t l^{\frac{2p}{p-5}}([T,\infty)\times\mathbb{Z})}$ is sufficiently small such that a fixed point $u^{(\alpha)} \in X_T$ can be attained. Flowing $u^{(\alpha)}$ backward in time, one obtains $f := u^{(\alpha)}(0)$. Hence the desired homeomorphism follows where the isometry is due to the conservation of $l^2$ norm.

Here it is shown that the asymptotic-state map $\alpha \mapsto u_+^{(\alpha)} \in l^2(\mathbb{Z})$, with a fixed initial datum, is continuous. From \eqref{duhamel2},
\begin{equation*}
    u_+^{(\alpha)} = f - i\mu \int_0^{\infty} U^{(\alpha)}(-t^\prime)\left(|u^{(\alpha)}(t^\prime)|^{p-1}u^{(\alpha)}(t^\prime)\right)dt^\prime,
\end{equation*}
and hence it suffices to show the continuity of the integral. For a fixed $t^\prime \in (0,\infty)$, the integrand converges to $U(-t^\prime)\left(|v(t^\prime)|^{p-1}v(t^\prime)\right) \in l^2(\mathbb{Z})$ as $\alpha \rightarrow \infty$; the argument of \Cref{gwp_convergence} may be modified to show pointwise (in $t^\prime$) continuity in $\alpha$. One may adopt the argument of \cite[Theorem 4]{stefanov2005asymptotic} verbatim to derive the decay estimate \eqref{dispersive_est}. The fDNLS obeys the same Strichartz estimates as DNLS, and therefore the nonlinear decay for $\| f \|_{l^{\frac{5}{4}}} < \delta_2$, for some $\delta_2(\alpha_0)>0$, follows where $\alpha \geq \alpha_0$ guarantees the uniformity in $\alpha$. Applying \eqref{dispersive_est} and the unitarity, we have
\begin{equation}\label{duhamel3}
\int_0^\infty \| |u^{(\alpha)}|^{p-1}u^{(\alpha)}\|_{l^2} dt^\prime = \int_0^\infty \| u^{(\alpha)} \|_{l^{2p}}^p dt^\prime \leq \int_0^\infty \| u^{(\alpha)} \|_{l^{5}}^p dt^\prime \lesssim_{\alpha_0} \int_0^\infty (1+|t^\prime|)^{-\frac{p}{5}} \| f\|_{l^{\frac{5}{4}}}^p dt^\prime < \infty,    
\end{equation}
and therefore $\alpha \mapsto u_+^{\alpha}$ is continuous by the Dominated Convergence Theorem.

Lastly, \eqref{singular_limit} is proved by contradiction. Let $0 < \| f \|_{l^{\frac{5}{4}}} < \delta_2$. Let $v_+ \in l^2(\mathbb{Z})\setminus \{0\}$ be the asymptotic state corresponding to $f$ under the DNLS flow. Let $0< \epsilon \ll \| \widehat{v_+} \|_{L^{1/3}} $. By the contradiction hypothesis, there exists a sequence $\{\alpha_j\}_{j \in \mathbb{N}}$ such that $\alpha_j \xrightarrow[j \rightarrow \infty]{} \infty$ and $\sup\limits_{t \in [0,\infty),\ j \in \mathbb{N}}\| u^{(\alpha_j)}(t) - v(t)\|_{l^2} < \frac{\epsilon}{4}$. By continuity, we have $\| u_+^{(\alpha_j)} - v_+\|_{l^2} < \frac{\epsilon}{4}$ for all sufficiently large $j$ (say $j > j_0$ for some $j_0 \in \mathbb{N}$). Let $\alpha \in \{\alpha_j\}_{j > j_0}$ and $\alpha > \max (\alpha_0^*,\alpha_0)$ where $\alpha_0^* > 0$ is from \Cref{lem}. The scattering results imply that there exists $T>0$ such that if $t \geq T$, then
\begin{equation*}
    \| u^{(\alpha)}(t) - U^{(\alpha)}(t)u_+^{(\alpha)}\|_{l^2}, \| v(t) - U(t)v_+\|_{l^2} < \frac{\epsilon}{4}.
\end{equation*}
Note that this $T>0$ is uniform in $\alpha$, and depends only on $\epsilon>0$; the residual $\| u^{(\alpha)}(t) - U^{(\alpha)}(t) u_+^{(\alpha)}\|_{l^2}$ from \eqref{duhamel2} is bounded above by the RHS of \eqref{duhamel3} with the lower bound of the integral replaced by $T$. Then, the triangle inequality yields
\begin{equation*}
    \| (U^{(\alpha)}(t) - U(t))v_+\|_{l^2} - \| u_+^{(\alpha)} - v_+\|_{l^2} - \| u^{(\alpha)}(t) - U^{(\alpha)}(t)u_+^{(\alpha)}\|_{l^2} - \| v(t) - U(t)v_+\|_{l^2} \leq \| u^{(\alpha)}(t) - v(t)\|_{l^2}. 
\end{equation*}

Since $\alpha_j \xrightarrow[j \rightarrow \infty]{}\infty$, we may choose $\alpha$ such that $2^\alpha > T$. By \Cref{lem} and the triangle inequality above,
\begin{equation*}
    \| \widehat{v_+} \|_{L^{1/3}} \lesssim \| (U^{(\alpha)}(t_\alpha) - U(t_\alpha))v_+\|_{l^2} \leq \sup_{t \in [T,\infty)}\| (U^{(\alpha)}(t) - U(t))v_+\|_{l^2} < \epsilon,  
\end{equation*}
a contradiction.
\end{proof}

\section{Modulational instability of CW solutions}\label{MI}

Here we focus on periodic solutions in the direction of propagation (in $t$) that do not decay in space. Hence there is no localization, initially, but MI triggers the formation of coherent states, a process that requires a further study. For $A>0$, define $u_{n}^{cw} = Ae^{iA^2 t}$ as the continuous wave (CW) solution and consider $u_n = (A + \mu v_n (t))e^{iA^2 t}$ where $v_n(t) \in \mathbb{C}$ and $|\mu|\ll 1$. 

\begin{proposition}\label{modulational_instability}
    The CW solution is linearly unstable if and only if $\Phi :=$ $\epsilon\sum\limits_{m=1}^{\infty}J_{m}(1-\cos(km)) - A^2 < 0$. The zero set $\{\Phi = 0\}$ is the graph of $A = A(k,\alpha,\epsilon)$ where $A$ is analytic on $\mathbb{T} \setminus \{0\} \times (0,\infty) \times (0,\infty)$. Furthermore there exists $C=C(\alpha)>0$ such that
    \begin{equation*}
    A \sim \left(\epsilon C \delta(k)\right)^{\frac{1}{2}} \text{as}\ k \rightarrow 0,    
    \end{equation*}
    where
    \begin{equation*}
        \delta(k)=
        \begin{cases}
            |k|^\alpha, &\alpha \in (0,2),\\
            (-\log |k|)|k|^2, &\alpha = 2,\\
            |k|^2, &\alpha \in (2,\infty].
        \end{cases}
    \end{equation*}
\end{proposition}

\begin{figure}[h!]
\begin{subfigure}[h]{0.38\linewidth}
\includegraphics[width=\linewidth]{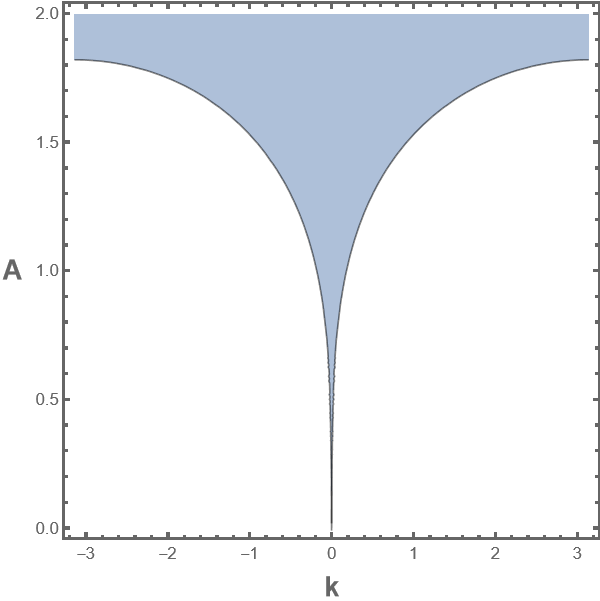}
\caption{$\alpha = 0.5$}
\end{subfigure}
\hfill
\begin{subfigure}[h]{0.38\linewidth}
\includegraphics[width=\linewidth]{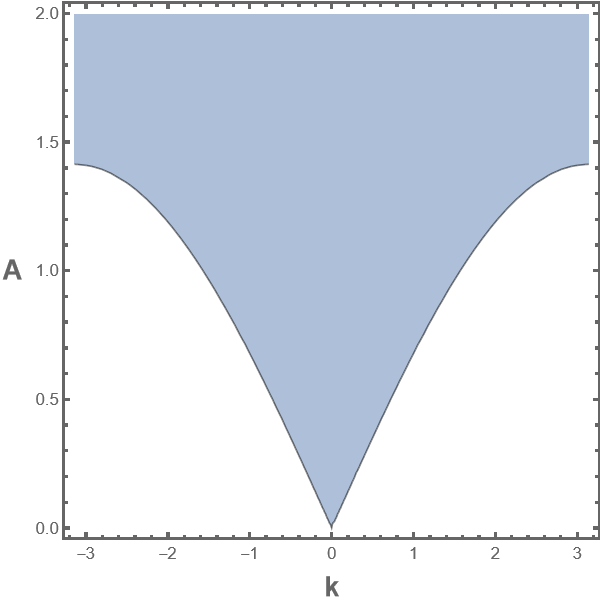}
\caption{$\alpha=50$}
\end{subfigure}

\begin{subfigure}[h]{0.38\linewidth}
\includegraphics[width=\linewidth]{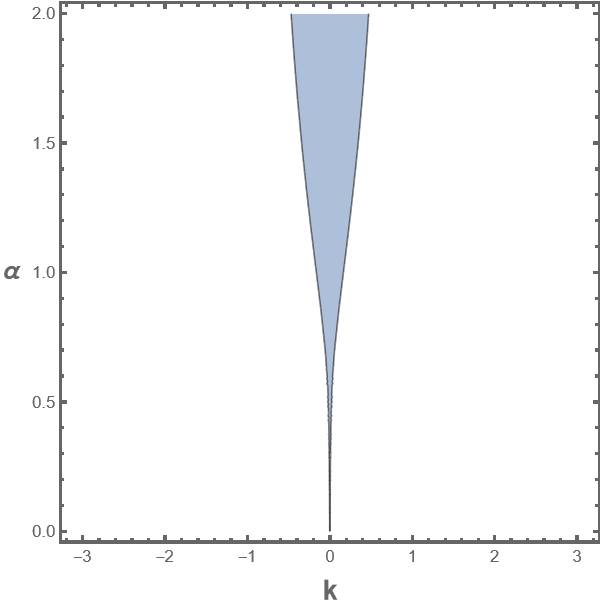}
\caption{$A = 0.5$}
\end{subfigure}
\hfill
\begin{subfigure}[h]{0.38\linewidth}
\includegraphics[width=\linewidth]{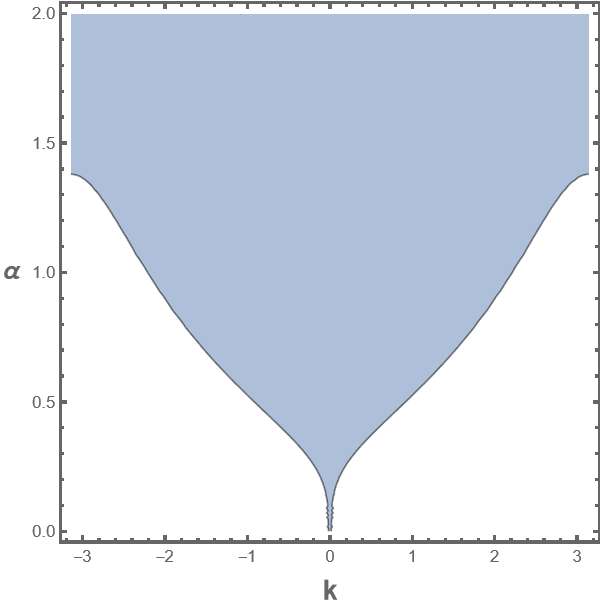}
\caption{$A=1.5$}
\end{subfigure}%
\caption{Instability regions (blue) for  $J_n = n^{-(1+\alpha)}$ in the $(k,A)$-plane and $(k,\alpha)$-plane where $k\in [-\pi,\pi]$ and $\epsilon = 1$.}
\label{fig:instability_region}
\end{figure}

\begin{proof}
The $\mathcal{O}(\mu)$ expansion yields
\begin{equation}\label{perturb}
i\partial_t v_n =  \epsilon\mathscr{L}_\alpha v_n  - A^{2}(v_n + \overline{v_n}),    
\end{equation}
and setting $v_n = f_n + ig_n := \Re v_n + i \Im v_n$, \eqref{perturb} is equivalent to
\begin{equation}\label{perturb2}
\frac{d}{dt}
\begin{pmatrix}
    f_n \\ g_n
\end{pmatrix}
=
\begin{pmatrix}
    0 & \epsilon \mathscr{L}_\alpha\\
    -\epsilon \mathscr{L}_\alpha + 2A^2 & 0
\end{pmatrix}
\begin{pmatrix}
    f_n \\ g_n
\end{pmatrix}.
\end{equation}
Taking the Fourier transform $F(k,t) = \mathcal{F}[f_n(t)](k),\ G(k,t) = \mathcal{F}[g_n(t)](k)$ and the ansatz
\begin{equation*}
\begin{pmatrix}
        F\\G
    \end{pmatrix}
    =
    \begin{pmatrix}
        P(k)\\Q(k)
    \end{pmatrix}
    e^{-i \Omega t},    
\end{equation*}
\eqref{perturb2} reduces to
\begin{equation}\label{perturb3}
    \begin{pmatrix}
        i\Omega & 2\epsilon \sum\limits_{m=1} J_m (1-\cos km)\\
        -\left(2\epsilon \sum\limits_{m=1} J_m (1-\cos km) - 2A^2\right) & i\Omega
    \end{pmatrix}
    \begin{pmatrix}
        P\\Q
    \end{pmatrix}
    =
    \begin{pmatrix}
        0\\0
    \end{pmatrix},
\end{equation}
where a nontrivial solution to \eqref{perturb3} exists if and only if $\Omega$ satisfies the dispersion relation
\begin{equation}\label{dispersion_relation}
    \Omega^2 = 4 \left (\epsilon \sum_{m=1}^{\infty}J_m (1-\cos(km)) \right ) \left (\epsilon\sum_{m=1}^{\infty}J_{m}(1-\cos(km)) - A^2 \right ).
\end{equation}
The frequency $\Omega$ is purely imaginary if and only if $\epsilon\sum\limits_{m=1}^{\infty}J_{m}(1-\cos(km)) < A^2$, leading to modulational instability.

Suppose $\Phi = 0$ and note that $\partial_A \Phi = -2A$. Since $A = 0$ if and only if $k=0$, $A(k,\alpha,\epsilon)$ is analytic whenever $k \neq 0$ by the analytic Implicit Function Theorem (IFT). In the neighborhood of $k=0$, \cite[Lemma A.1]{kirkpatrick2013continuum} implies that there exists $C(\alpha)>0$ such that
\begin{equation*}
    \frac{A^2}{\epsilon C \delta(k)} = \frac{A^2}{\epsilon C \delta(k)}\cdot \frac{\sum\limits_{m=1}^{\infty}J_{m}(1-\cos(km))}{\sum\limits_{m=1}^{\infty}J_{m}(1-\cos(km))} = \frac{\sum\limits_{m=1}^{\infty}J_{m}(1-\cos(km))}{C \delta(k)} \xrightarrow[k\rightarrow 0]{}1.  
\end{equation*}
\end{proof}

When $J_n = n^{-(1+\alpha)}$, \Cref{fig:instability_region} shows the region of instability. Note the non-analytic dependence of $A,\alpha$ on $k$ near the zero wavenumber. The instability region expands for bigger values of $A$, which is consistent with the $-A^2$ term in \eqref{dispersion_relation}.

To compute the value(s) of $k \in [-\pi,\pi]$ that maximizes the exponential gain of the modulational instability, the RHS of \eqref{dispersion_relation} needs to be minimized under the condition $\Phi < 0$; that a minimum exists follows from the Extreme Value Theorem since $\Omega(k)^2$ is continuous by the definition of $(J_n)_{n=1}^\infty$. An explicit computation that minimizes $\Omega(k)^2$ is shown for a specific interaction kernel. 

\begin{figure}[ht]
\centering
\includegraphics[width = 0.3\textwidth]{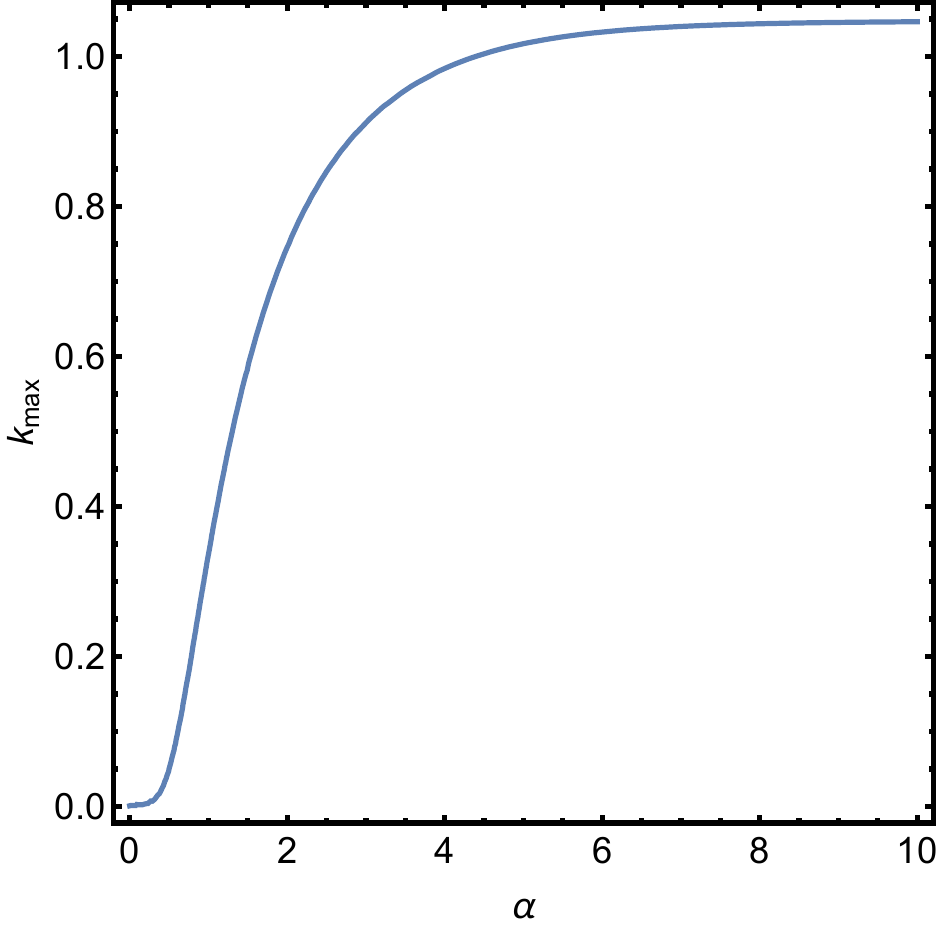}
\caption{$k_{max}$ as a function of $\alpha$ when $J_n = n^{-(1+\alpha)}$. The values $A=1, \epsilon =1$ are used to generate this plot. For such parameters, $A_0 \geq 2 > A$ where $A_0$ is defined in \Cref{cor:2.1}, and therefore \eqref{wavenumber} is used to solve $k_{max}$ in $\alpha$. Observe that small $\alpha$ gives small $k_{max}$, manifested in the top left plot of \Cref{fig:5}. That $k_{max}$ has an upper bound is numerically verified in \Cref{fig:5} for increasing values of $\alpha$.}
\label{fig:k_vs_alpha}
\end{figure}

\begin{corollary}\label{cor:2.1}
For $\alpha,\epsilon >0$, let $J_n = n^{-(1+\alpha)},\ \Tilde{w}(k) = \sum\limits_{m=1}^{\infty} \frac{1-\cos km}{m^{1+\alpha}},$ and $A_0 := \left(4\epsilon (1-2^{-(1+\alpha)})\zeta(1+\alpha)\right)^{\frac{1}{2}}$. Then for any $A \geq A_0$, 
\begin{equation*}\label{min_largeA}
    \min_{k \in [-\pi,\pi]}\Omega(k)^2 = \Omega(\pm \pi)^2 = -A_0^2 (2A^2 - A_0^2),
\end{equation*}
and if $0 < A < A_0$, then there exists a unique $k_{max} \in (0,\pi)$ that minimizes $\Omega^2$ where $k_{max}$ satisfies
\begin{equation}\label{wavenumber}
\Tilde{w}(k_{max}) = \frac{A^2}{2\epsilon}    
\end{equation}
and
\begin{equation*}\label{min_smallA}
    \min_{k \in [-\pi,\pi]}\Omega(k)^2 = \Omega(\pm k_{max})^2 = -A^4.
\end{equation*}
\end{corollary}

\begin{proof}
Since $\Omega(k)^2$ is even and $\Omega(0) = 0$, let $k \in (0,\pi]$ without loss of generality. By direct computation, it can be shown that $\Tilde{w}(0)=0,\ \Tilde{w} \in C^1(0,\pi)$, and $\Tilde{w}$ is increasing on $(0,\pi)$. From the derivative, we have
\begin{equation*}\label{deriv}
    \begin{split}
    \frac{d}{dk} \left(\Omega(k)^2\right) &= 4\epsilon \frac{d}{dk}\Tilde{w}(k) \left( 2\epsilon \Tilde{w}(k) - A^2 \right ),
    \end{split}
\end{equation*}
and therefore $\Omega^2$ is decreasing if and only if $2\epsilon \Tilde{w}(k) - A^2 \leq 0$. Since
\begin{equation}\label{zeta_odd}
    \max_{k \in [-\pi,\pi]} \Tilde{w}(k) = \Tilde{w}(\pm\pi) = \sum_{m \geq 1,\ m\ \text{odd}} \frac{2}{m^{1+\alpha}}= 2(1-2^{-(1+\alpha)})\zeta(1+\alpha),
\end{equation}
if $A \geq A_0$, then $2\epsilon \Tilde{w}(k) - A^2 \leq 0$ for all $k \in (0,\pi]$. Hence $\Omega^2$ is minimized at $k = \pm \pi$ and the value of $\Omega(\pm \pi)^2$ follows from \eqref{zeta_odd}. If $0 < A < A_0$, then there exists a unique $k_{max} \in (0,\pi)$ satisfying $\Tilde{w}(k_{max}) = \frac{A^2}{2\epsilon}$ by the Intermediate Value Theorem and
\begin{equation*}
    \Omega(k_{max})^2 = 4\epsilon \Tilde{w}(k_{max}) \left(\epsilon\Tilde{w}(k_{max}) - A^2\right) = -A^4.
\end{equation*}
\end{proof}
What dynamics emerges after the initial linear growth due to MI is a question that has been posed since the studies of the FPUT chain. Based on many studies in discrete and continuous models, it is expected that the discrete modulational instability is the potential mechanism for the formation of nonlinear localized modes, such as discrete breathers and envelope solitons. What is not known is which type emerges in the fNLSE and what the role of the parameter $\alpha$ is. Figure 3 shows a sample of emerging patterns triggered by numerical noise. While at this time we will not examine these emerging patterns, it is clear that $\alpha$-dependent  coherent and robust patterns develop.
\begin{figure}[H]
\begin{subfigure}[h]{0.52\linewidth}
\includegraphics[width=\linewidth]{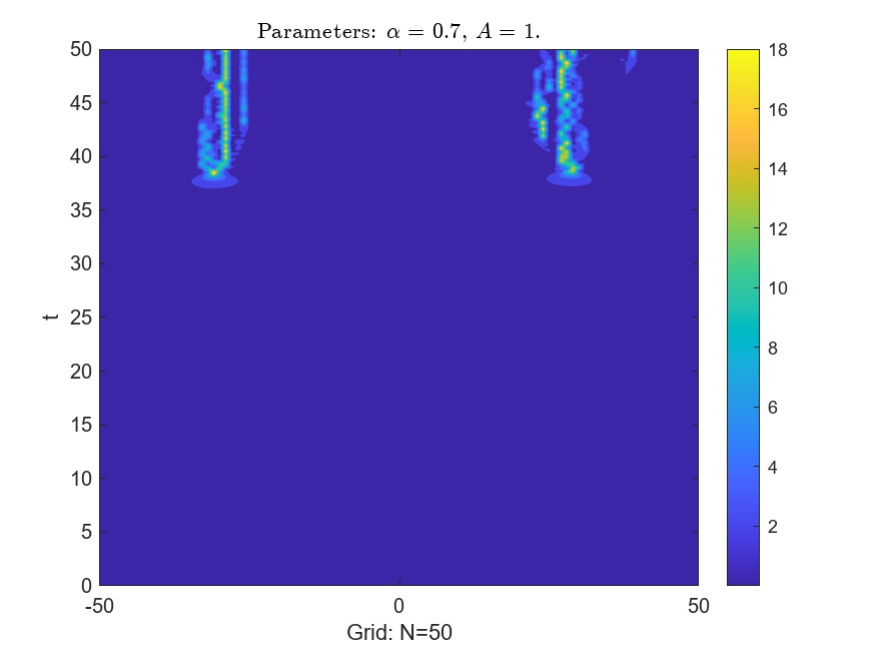}
\end{subfigure}
\begin{subfigure}[h]{0.52\linewidth}
\includegraphics[width=\linewidth]{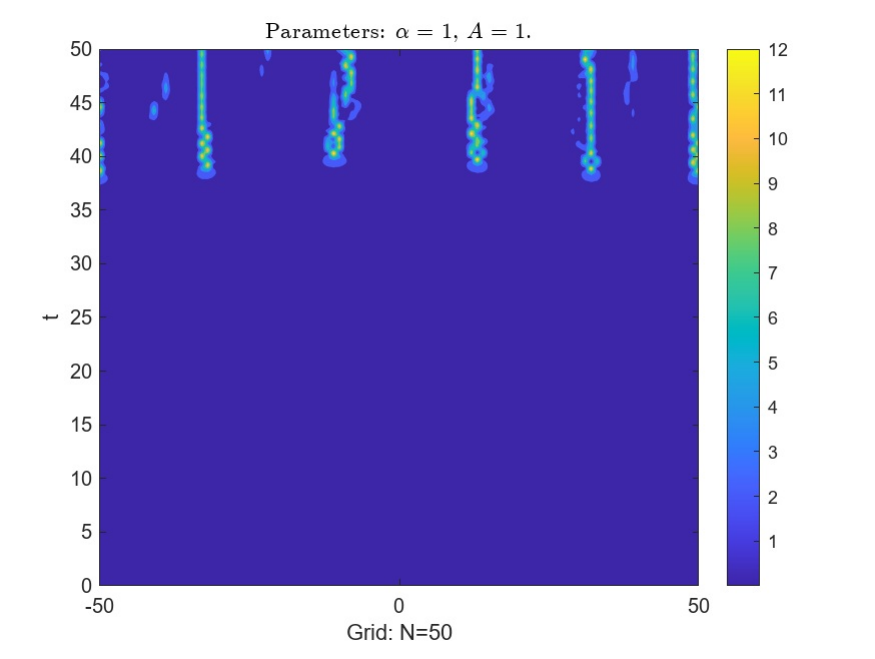}
\end{subfigure}
\begin{subfigure}[h]{0.52\linewidth}
\includegraphics[width=\linewidth]{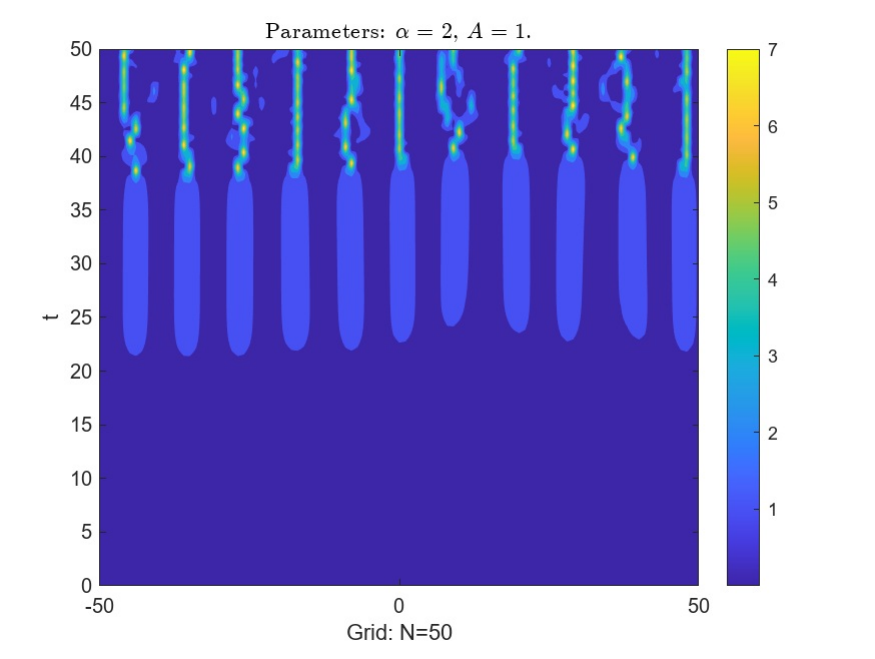}
\end{subfigure}%
\begin{subfigure}[h]{0.52\linewidth}
\includegraphics[width=\linewidth]{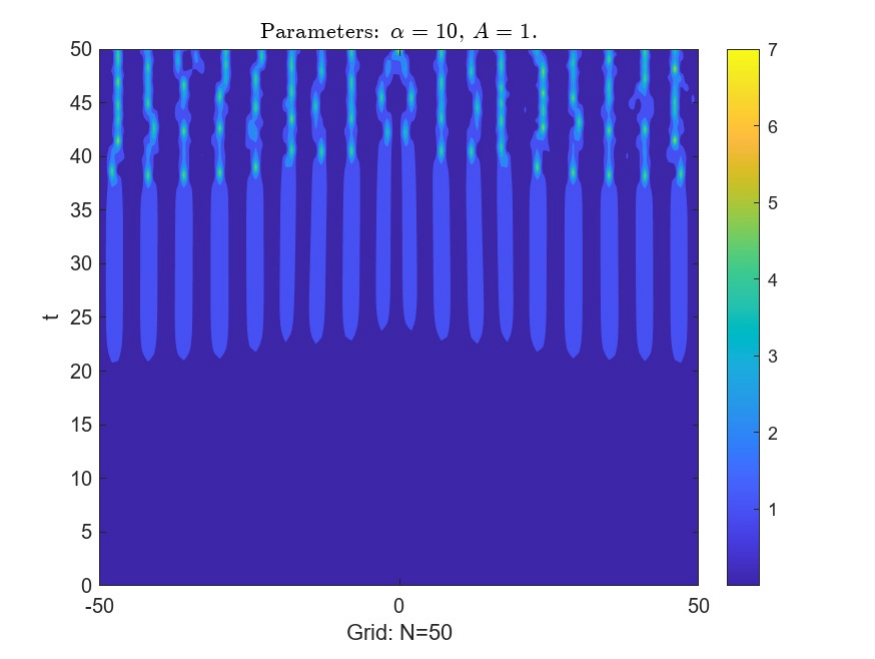}
\end{subfigure}
\caption{Emerging patterns triggered by numerical noise. Background intensity is the same in all cases, $A=1$ . The nonlinear state should be quasiperiodic, with an $\alpha$ dependence on the separation between peaks, likely to be close to the most unstable MI wavenumber $k_{max}$.}
\label{fig:5}
\end{figure}

\section{Asymptotic construction of stationary solutions}\label{solution_DNLS}

We study stationary, localized solutions of the lattice NLS in the anti-continuum (AC) regime and compare the onsite (mode A; centered on a lattice site; $q_n = q_{-n}$) and offsite (mode B; centered between two sites; $g_n = g_{-(n-1)}$) branches. Replacing $Q$ by $\sqrt{\omega} Q$, \eqref{generalmodel_stationary} reduces to $-Q_n = \rho \mathscr{L} Q_n - Q_n^3$ where $\rho := \frac{\epsilon}{\omega}$ is a small parameter. We call AC regime any path with $\rho \rightarrow 0$. We present two physically relevant scenarios. When $\epsilon \rightarrow 0$ at fixed $\omega$ (classical AC), the phonon band of $\epsilon \mathscr{L}$, the unscaled operator, has width $O(\epsilon)$ and collapses to a point with the temporal frequency $\omega$ lying outside it; when $\omega \rightarrow \infty$ at fixed $\epsilon$ (high frequency limit), $\omega$ is outside the $O(\epsilon)$ phonon band. At $\rho = 0$, let $Q^*_{A} := \delta_0$ and $Q^*_{B} := \delta_0 + \delta_1$ where $\delta_i$ is the Kronecker delta at site $i$. At those AC seeds, our Hamiltonian model satisfies the non-resonance and non-degeneracy conditions in the terminology of Mackay-Aubry \cite{mackay1994proof}, which proved the AC continuation in a general setting. Our contribution is constructive and model-specific with the following goals: 1. derive an explicit non-Markovian (infinite-memory) recurrence for fDNLS on/offsite profiles by adapting the map approach in \cite[Chapter 11]{kevrekidis}; 2. prove these sequences solve the exact fDNLS asymptotically as $\rho \rightarrow 0$; 3. use IFT to obtain the exact branches from the AC seeds and show the explicit sequences are within the $O(\rho)$-neighborhood in $l^2(\mathbb{Z})$; 4. state the PNB formula and analyze it in the AC regime; 5. present numerics, including parameter values where PNB vanishes, and evidence of traveling waves.

\subsection{Localized solutions for DNLS}\label{localized}

This subsection records the minimal AC template used later for the nonlocal model. To motivate \Cref{dnls_asymptotic}, consider \eqref{generalmodel_stationary} with $J_n = 1$, if $n=1$, and $J_n = 0$ otherwise. Assume $q_n \geq 0$ and $q_{n+1} \ll q_n$ for all $n \geq 0$. Neglecting $q_{\pm 1}$ yields
\begin{equation*}
\omega q_0 = -2\epsilon q_0 + q_0^3,    
\end{equation*}
and therefore $q_0 = \sqrt{\omega+2\epsilon}$. For $n \geq 1$, further assume $q_n \ll 1$, after which \eqref{generalmodel_stationary} neglecting $q_{n+1}$ and the nonlinear term yields
\begin{equation*}
    \omega q_n = \epsilon (q_{n-1} -2q_n),
\end{equation*}
and therefore $q_n = \frac{\epsilon}{(\omega+2\epsilon)} q_{n-1}$. By this method, a forward map is defined where $q_n$ depends only on $q_{n-1}$; a similar computation can be done with $\{g_n\}$. This approximation method is comparable to the fixed point approach using the Green's function in \cite{cuevas2008approximation}. This explicit construction yields a sequence $\{q_n\} \in l^2(\mathbb{Z})$ that satisfies \eqref{dnls} asymptotically in the AC regime. The proof is in \Cref{appendix}.
\begin{proposition}\label{dnls_asymptotic}
For $\omega,\epsilon>0$, let $\{q_n\}$ satisfy
\begin{equation}\label{dnls4}
q_n = \left(\frac{\epsilon}{\omega+2\epsilon}\right)^{|n|} \sqrt{\omega+2\epsilon},\ n \in \mathbb{Z}.
\end{equation}
Then $q_{n+1} = o(q_n)$ as $\rho \rightarrow 0$. Moreover, $\{q_n\}$ satisfies \eqref{generalmodel_stationary} asymptotically as $\rho \rightarrow 0$ uniformly in $n \in \mathbb{Z}$, or more precisely,
\begin{equation}\label{dnls_asymp}
\lim_{\rho \rightarrow 0} \sup_{n \in \mathbb{Z}}\left|\frac{\omega q_n}{\epsilon(q_{n+1}+q_{n-1}-2q_{n}) + q_n^3} - 1\right|= 0.
\end{equation}
\end{proposition}

\begin{remark}
    For $n \geq 1$, define $\{g_n\}$ as
\begin{equation}\label{dnls5}
    g_n = \left(\frac{\epsilon}{\omega+2\epsilon}\right)^{n-1} \sqrt{\omega+\epsilon},
\end{equation}
and extend to $\mathbb{Z}$ by $g_{n} = g_{-(n-1)}$. Then $g_{n+1} = o(g_n)$ as $\rho \rightarrow 0$ for all $n \geq 1$ and satisfies \eqref{dnls_asymp} with $q_n$ replaced by $g_n$.
\end{remark}

A direct computation via geometric series and \eqref{energy} yields an explicit expression for $E_{DNLS}$. Recall that $E_A,E_B$ are energies of $q_n,g_n$, respectively.

\begin{proposition}\label{energy_dnls_onoff}
Let $q_n,g_n$ be given by \eqref{dnls4}, \eqref{dnls5}, respectively. Then,
\begin{equation}\label{energy_dnls}
    \begin{split}
        E_{A}(\epsilon,\omega) &= \frac{\epsilon (\omega+\epsilon) (\omega+2\epsilon)}{\omega+3\epsilon} - \frac{(\omega+2\epsilon)^2 ((\omega+2\epsilon)^4+\epsilon^4)}{4((\omega+2\epsilon)^4-\epsilon^4)},\\
        E_{B}(\epsilon,\omega) &= \frac{\epsilon (\omega+\epsilon)^2}{\omega+3\epsilon} - \frac{(\omega+\epsilon)^2}{2\left(1-\frac{\epsilon^4}{(\omega+2\epsilon)^4}\right)}.
    \end{split}
\end{equation}
\end{proposition}

Suppose $q_n$ oscillates in time at $\omega_A$ and $g_n$ at $\omega_B$. If $q_n$ and $g_n$ are two modes of the traveling wave solution, then $N_A = N_B$. Since $N_A \sim \omega_A,\ N_B \sim 2\omega_B$ for $\rho_A,\rho_B \ll 1$, assume $\omega_A = 2\omega_B$. As \cite[Equation 9]{KivCam}, the Peierls-Nabarro barrier is defined as the energy difference, at fixed power, of the two modes at $\omega_A, \omega_B$, which are given by \eqref{energy_dnls}.

\begin{corollary}\label{pnb_dnls}
Let $\Delta E_{AB} = E_A(\epsilon,\omega_A) - E_B(\epsilon,\omega_B)$ where $\omega_A = 2\omega_B$. Setting $\epsilon = k \omega_A$ for $k>0$, we have
\begin{equation*}
    \Delta E_{AB} = -\gamma(k) \omega_A^2,
\end{equation*}
where $\gamma(k)$ is a strictly positive rational function satisfying $\gamma(k) \xrightarrow[k\rightarrow 0+]{}\frac{1}{8}$ and $\inf\limits_{k \geq 0} \gamma(k) >0$. Therefore $\Delta E_{AB}<0$ for any $\epsilon,w_A>0$ and $\lim\limits_{(\epsilon,w_A)\rightarrow (\epsilon^*,w_A^{*})} \Delta E_{AB} = 0$ if and only if $(\epsilon^*,w_A^{*}) = (0,0)$. Furthermore as $w_A \rightarrow \infty$,
\begin{equation*}
\Delta E_{AB} \sim_\epsilon -\frac{w_A^2}{8},\ \frac{E_{A}(\epsilon,w_A)}{E_{B}(\epsilon,w_B)} \sim_\epsilon 2.
\end{equation*}
\end{corollary}

\subsection{Localized solutions for fDNLS}\label{fDNLS}

We study stationary solutions of fDNLS at fixed lattice spacing,
\begin{equation*}
-\omega {q}_n = \epsilon \sum_{m\in\mathbb{Z} \setminus \{n\} } \frac{q_n - q_m}{|n - m|^{1+\alpha}} - q_n^3,\ (n,t) \in \mathbb{Z} \times \mathbb{R},
\end{equation*}
while keeping the full discrete symbol $\sigma_\alpha(k)$ in \eqref{discrete_symbol}, which has a non-analytic expansion at $k=0$. This branch term forces an algebraic tail in the Green’s function and, consequently, algebraic spatial decay of on/offsite profiles (see \Cref{uniform_convergence}). By contrast, statements of exponential decay for $\alpha > 2$ (see \cite{PhysRevE.55.6141}) arise when $\sigma_\alpha(k)$ is replaced by its analytic $k^2$ Taylor expansion, a long-wave approximation that captures the core but is not uniform in the far-field on the fixed lattice; similarly in \cite{flach1998breathers}, the authors showed, via a similar nonlocal model, the existence of the transitional lattice site $n_c \in \mathbb{N}$ such that the solution decays algebraically in the far-field regime $|n| \gg n_c$ and $\alpha > 2$. 

We construct and justify nonlocal on/off-site profiles, prove $O(\rho)$ proximity to the exact branches, and use them to analyze the PNB and its numerical consequences. To this end, define
\begin{subnumcases}{}
   q_0 = \sqrt{\omega_A + 2 \epsilon J},\ q_n = \frac{\epsilon\left(J_n q_0 + \sum\limits_{m=1}^{n-1} \left(J_{n-m} + J_{n+m}\right)q_m\right)}{\epsilon\left(2J - J_{2n}\right) +\omega_A}, & $n \geq 1$, \label{dnls2}
   \\
   g_0 = \sqrt{\omega_B+\epsilon(2J - J_1)},\ g_n = \frac{\epsilon\sum\limits_{m=1}^{n-1} \left(J_{n-m} + J_{n+m-1}\right)g_m}{\epsilon\left(2J - J_{2n-1}\right) + \omega_B}, & $n \geq 2$, \label{fdnls5}
\end{subnumcases}
where $J = \sum\limits_{n=1}^{\infty}J_{n}$. To provide an insight for the definitions above, the first two terms of $q_n$ are shown explicitly. For $q_0$, neglect $q_n$ for $|n| \geq 1$. Then
\begin{equation*}
-\omega_A q_0 = 2\epsilon J q_0 - q_0^3,
\end{equation*}
and hence $q_0$. For $q_1$, neglect the nonlinear term $q_1^3$ and $q_n$ where $|n| \geq 2$. Since $q_1 = q_{-1}$,
\begin{equation*}
-\omega_A q_1 = \epsilon\sum_{m \neq 1} J_{|1-m|}(q_1 - q_m) = 2\epsilon J q_1 - \epsilon\left(J_1 q_0 + J_2 q_1\right),
\end{equation*}
and hence $q_1$. Although the analytic description of the asymptotic sequences is not immediately tractable, they are given by the power series expansion when $\rho \ll 1$.
\begin{proposition}\label{fdnls_asymptotic}
Let $0 < \rho < \frac{1}{2J}$. Then there exist sequences $\{c_{nk}\}_{k \geq 1},\{d_{nk}\}_{k \geq 1} \subseteq \mathbb{R}$ such that
\begin{equation}\label{induction}
\begin{split}
\frac{q_n}{q_0} &= \sum_{k=1}^\infty c_{nk}\rho^k,\ n \geq 1\\  \frac{g_n}{g_0} &= \sum_{k=1}^\infty d_{nk}\rho^k,\ n\geq 2,  
\end{split}
\end{equation}
where the series are absolutely convergent and $c_{n1} = J_n,\ d_{n1} = J_{n-1} + J_n$.
\end{proposition}

\begin{proof}
For brevity, our proof concerns $q_n$ only. Let 
\begin{equation*}
\gamma_n = \frac{\rho}{1+\rho(2J-J_{2n})}.    
\end{equation*}
Since $\rho$ is sufficiently small by hypothesis, $\gamma_n = \rho + O_\alpha(\rho^2)$ by series expansion where the error term is uniform in $n$ since $|\rho(2J-J_{2n})| < 2 \rho J <1$. Let $n=1$. From \eqref{dnls2},
\begin{equation*}
\frac{q_1}{q_0} = \gamma_1 J_1 = \frac{\rho J_1}{1+\rho (2J - J_1)},   
\end{equation*}
and hence \eqref{induction} with $c_{n1} = J_1$. Suppose the claim holds for $m = 1,\dots,n-1$. Then
\begin{align}\label{taylorseries}
\frac{q_n}{q_0} &= \gamma_n J_n + \gamma_n \sum\limits_{m=1}^{n-1} \left(J_{n-m} + J_{n+m}\right)\frac{q_m}{q_0}\nonumber\\
&= \gamma_n J_n + \gamma_n \sum_{k=1}^\infty \sum_{m=1}^{n-1} \left(J_{n-m} + J_{n+m}\right) c_{mk} \rho^k, 
\end{align}
where the series is absolutely convergent since
\begin{equation*}
\sum_{k=1}^\infty \sum_{m=1}^{n-1} \left(J_{n-m} + J_{n+m}\right) |c_{mk}| \rho^k \leq J  \cdot \max_{1 \leq m \leq n-1} \left(\sum_{k=1}^\infty |c_{mk}| \rho^k \right) < \infty.   
\end{equation*}
The second term of \eqref{taylorseries} is $O(\rho^2)$, and hence the dominant term of $\frac{q_n}{q_0}$ is $\rho J_n$.   
\end{proof}

By \Cref{fdnls_asymptotic},
\begin{equation*}
\left|\frac{q_{n}/q_{0}}{\rho J_n} - 1 \right| = J_n^{-1} \left| \sum\limits_{k=2}^\infty c_{nk} \rho^{k-1}\right|  \xrightarrow[\rho \rightarrow 0]{} 0,    
\end{equation*}
and hence for any $n \in \mathbb{Z} \setminus \{0\}$,
\begin{equation}\label{convergence_rate}
    \frac{q_n}{q_0} \sim \rho J_n,\ \rho \rightarrow 0.
\end{equation}
\begin{proposition}\label{uniform_convergence}
The convergence rate of \eqref{convergence_rate} is uniform in $n$. More precisely, for all $\epsilon_1,\alpha >0$, there exists $\rho_* = \rho_*(\epsilon_1,\alpha) > 0$ such that for any $0< \rho < \rho_*$,
\begin{equation}\label{uniformbound}
\begin{split}
\frac{1}{1+\epsilon_1} &< \frac{q_n/q_0}{\rho J_{n}} < 1+\epsilon_1,\ n \geq 1,\\
\frac{1}{1+\epsilon_1} &\leq \frac{g_n/g_0}{\rho (J_{n} + J_{n-1})} \leq 1+\epsilon_1,\ n \geq 2.
\end{split}
\end{equation}
\end{proposition}
\begin{proof}

The proof is for $q_n$ without loss of generality. Since $n^{1+\alpha} J_n \xrightarrow[n\rightarrow \infty]{} A$, there exists $N \in \mathbb{N}$ such that for any $n \geq N$, we have $|J_n - \frac{A}{n^{1+\alpha}}|<\frac{A}{2n^{1+\alpha}}$. Define
\begin{equation*}
    \rho_* = \min\left(\frac{\epsilon_1}{2J},\min\limits_{2 \leq n \leq 2N}\left(\frac{\epsilon_1 J_n}{2(1+\epsilon_1)(n-1)J^2}\right), \frac{\epsilon_1}{3(1+2^{2+\alpha})(1+\epsilon_1)J}\right).
\end{equation*}

By \eqref{dnls2},
\begin{equation*}
    \frac{q_n}{q_0} \geq \gamma_n J_n > \frac{\rho J_n}{1+2\rho J} > \frac{\rho J_n}{1+\epsilon_1}, 
\end{equation*}
and hence the lower bound of \eqref{uniformbound}. The proof for the upper bound is by induction. For the base case, we have
\begin{equation*}
    \frac{q_1}{\rho J_1 q_0} = \frac{1}{1+\rho(2J-J_2)} <   1+\epsilon_1.
\end{equation*}

Let $C = (1+\epsilon_1)\rho q_0$. For $2 \leq n \leq 2N$,
\begin{equation*}
    q_n < \rho\left(J_n q_0 + C \sum_{m=1}^{n-1} (J_{n-m}+J_{n+m})J_m\right) \leq \rho \left( J_n q_0 + 2C(n-1)J^2\right) < C J_n,
\end{equation*}
since $\rho < \rho_*$. 

Let $n > 2N$. Suppose $q_m < C J_m$ for all $1 \leq m \leq n-1$. Observe that
\begin{equation*}
\begin{split}
\sum_{m=1}^{n-1} J_{n-m}q_m &< C\left(\sum_{1 \leq m \leq \frac{n}{2}} J_{n-m}J_m + \sum_{\frac{n}{2} < m \leq n-1} J_{n-m}J_m\right)\\
&< \frac{3CA}{2} \left(\sum_{1 \leq m \leq \frac{n}{2}} \frac{J_m}{(n-m)^{1+\alpha}} + \sum_{\frac{n}{2}<m\leq n-1} \frac{J_{n-m}}{m^{1+\alpha}}\right)\\
&\leq 3 \cdot 2^{1+\alpha} CAn^{-(1+\alpha)}J < 3 \cdot 2^{2+\alpha} C J J_n.
\end{split}
\end{equation*}
A similar computation yields \begin{equation*}
    \sum_{m=1}^{n-1} J_{n+m} q_m < 3CJ J_n,
\end{equation*}
and altogether,
\begin{equation*}
    q_n < \rho \left(q_0 + 3(1+2^{2+\alpha})C J \right) J_n < C J_n.
\end{equation*}
This completes the induction, and the claim for $g_n$ follows similarly.   
\end{proof}
\begin{corollary}\label{cor}
    Assuming the hypotheses of \Cref{uniform_convergence},
    \begin{equation*}    
    \frac{q_n}{q_0} \sim \frac{A\rho}{n^{1+\alpha}},\hspace{20pt} \frac{q_{n+1}}{q_n} \sim \left(\frac{n}{n+1}\right)^{1+\alpha},\hspace{20pt} \frac{q_{2n}}{q_n} \sim 2^{-(1+\alpha)},
    \end{equation*}
as $\rho \rightarrow 0$ and uniformly in $n \in \mathbb{Z}$, and similarly for $\{g_n\}$.
\end{corollary}
As \Cref{dnls_asymptotic}, it is shown that \eqref{dnls2}, \eqref{fdnls5} solve \eqref{generalmodel_stationary} in an asymptotic sense.
\begin{proposition}\label{fdnls_asymptotic2}
For any $\epsilon,\omega_A,\alpha>0$, define $\{q_n\}$ by \eqref{dnls2}. Then $q_{n+1} = O(q_n)$ as $\rho \rightarrow 0$, and
\begin{equation}\label{fdnls_asymp}
\lim_{\rho \rightarrow 0} \sup_{n \in \mathbb{Z}}\left| \frac{-\omega_A q_n}{\epsilon \mathscr{L}_\alpha q_n - q_n^3} - 1\right| = 0,
\end{equation}
and similarly for $\{g_n\}$.    
\end{proposition}
\begin{proof}
The big-O bound follows from \Cref{cor}. The proof is for $q_n$ and $n \geq 1$. By \Cref{uniform_convergence},
\begin{equation*}
-\omega_A q_n \sim -\omega_A \rho J_n q_0 = -\omega_A^{\frac{3}{2}}\rho J_n (1+2\rho J)^{\frac{1}{2}}.   
\end{equation*}
The interaction term is given by
\begin{equation*}
\epsilon\sum_{m \neq n} J_{|n-m|}(q_n - q_m) =  \epsilon J_{n}(q_n - q_0) + \epsilon\sum_{m \notin \{0,n\}} J_{|n-m|}(q_n - q_m) =: I + II,
\end{equation*}
and estimating the terms separately, we have
\begin{equation*}
I \sim -\epsilon J_n q_0 = - \omega_A^{\frac{3}{2}} \rho J_n (1+2\rho J)^{\frac{1}{2}},
\end{equation*}
and
\begin{equation}\label{asymp_estimate}
II \sim  \epsilon \rho q_0 \sum_{m \notin \{0,n\}} J_{|n-m|} (J_{|n|} - J_{|m|}) \leq  2 \omega_A^{\frac{3}{2}} \rho^2 (1+2\rho J)^{\frac{1}{2}}J^2.    
\end{equation}
It follows that $II$ is $O(\rho^2)$ and hence negligible. So is the nonlinear term, which is $O(\omega_A^{\frac{3}{2}} \rho^{3})$. The uniformity in $n$ follows as \Cref{cor}. The limit \eqref{fdnls_asymp} can be shown similarly for the case $n=0$. 
\end{proof}
By substituting $Q \mapsto \sqrt{\omega} Q$, the reduction of \eqref{generalmodel_stationary} and its derivative in the state variable are
\begin{equation}\label{scaled_gs}
\begin{split}
F(Q,\rho) &:= (I + \rho \mathscr{L})Q - Q^3 = 0,\\
D_{Q}F(Q,\rho) &= I + \rho \mathscr{L} - 3Q^2,    
\end{split}
\end{equation}
where the notation $Q^3$ denotes pointwise multiplication and $F: l^2(\mathbb{Z};\mathbb{R}) \times \mathbb{R} \rightarrow l^2(\mathbb{Z};\mathbb{R})$. Relevant to our application are two AC seeds: $Q^{*}_{A} = \delta_0$ and $Q^{*}_{B} = \delta_0 + \delta_1$ where $\delta_i$ is the Kronecker delta at $i \in \mathbb{Z}$. Since $D_{Q}F(Q(0),0) = I - 3 Q(0)^2$ is invertible where $Q(0) \in \{Q_A^{*},Q_B^{*}\}$, there exists a unique $C^1$-curve $\rho \mapsto Q(\rho)$ that satisfies $F(Q(\rho),\rho) = 0$ locally at $\rho = 0$ by IFT. The implication of \eqref{anti_continuum} is that the sequences \eqref{dnls2}, under an appropriate normalization, approximate the exact AC branches within $O(\rho)$. Though \Cref{ift} is stated for onsite solutions, the analogous result for offsite solutions also holds.
\begin{proposition}\label{ift}
Let $Q^{(\rho)} := \omega_A^{-\frac{1}{2}}q^{(\rho)}$ where $q^{(\rho)}$ is given by \eqref{dnls2}. For any $0 \leq \rho < \frac{1}{16J}$, there exists $C>0$ independent of $\alpha > 0$ ($C = \sqrt{22}+16$ suffices) such that 
\begin{equation}\label{anti_continuum}
\| Q^{(\rho)} - Q(\rho) \|_{l^2} \leq C J \rho.    
\end{equation}
\end{proposition}
\begin{proof}
By \eqref{dnls2} and \Cref{uniform_convergence},
\begin{equation*}
Q^{(\rho)}_0 - (Q^*_A)_{0} = \frac{2J\rho}{\sqrt{1+2J\rho} + 1},\ |Q^{(\rho)}_n| < 2\sqrt{1+2J\rho} ( J_{|n|}\rho),    
\end{equation*}
respectively, and therefore
\begin{equation}\label{residual}
    \| Q^{(\rho)} - Q^*_{A} \|_{l^2}^2 < 4J^2 \rho^2 + 4(1+2J\rho) \rho^2 \sum_{n \neq 0} J_{|n|}^2 < 22 J^2 \rho^2.
\end{equation}

We obtain  bounds on $Q^{\prime} := \frac{dQ}{d\rho}$ via IFT. The operator norm $\| D_Q F (Q_A^{*},0)^{-1}\| = 1$ since $D_Q F (Q_A^{*},0)$ has simple eigenvalues at $-2,1$. There exist $r_1, r_2 > 0$ such that whenever $\| Q - Q_A^{*} \|_{l^2} < r_1,\ \rho < r_2$, we have
\begin{equation*}
\begin{split}
\| D_Q F(Q,\rho) - D_Q F (Q_A^{*},0)\| &= \| \rho \mathscr{L} - 3(Q^2 - (Q_A^{*})^2)\|\\
&\leq \rho \| \mathscr{L} \| + 3(\| Q \|_{l^\infty} + \| Q_A^{*} \|_{l^\infty}) \| Q - Q_A^{*} \|_{l^2} < 4J r_2 + (6+3r_1) r_1 \leq \frac{1}{2},
\end{split}
\end{equation*}
where $r_1 = \frac{1}{100},\ r_2 = \frac{1}{16J}$ satisfy the last inequality. By the Neumann series,
\begin{equation}\label{neumann}
\| D_Q F(Q,\rho)^{-1}\| < 2    
\end{equation}
for any $(Q,\rho)$ sufficiently close to $(Q_A^{*},0)$ within the $(r_1,r_2)$-rectangle. The $\rho$-derivative of $F(Q(\rho),\rho) = 0$ yields
\begin{equation*}
    D_Q F (Q(\rho),\rho) Q^{\prime}(\rho) = -\mathscr{L} Q,
\end{equation*}
and therefore by \eqref{neumann} and $\| Q(\rho) - Q_A^{*} \|_{l^2} < r_1$, we have
\begin{equation}\label{neumann2}
\| Q^{\prime}(\rho) \|_{l^2} \leq \| D_Q F (Q(\rho),\rho)^{-1}\| \| \mathscr{L} \| \| Q(\rho) \|_{l^2} < 2 (4J) (1+r_1) \leq 16J.   
\end{equation}
By the mean value theorem and \eqref{neumann2},
\begin{equation*}
    \| Q(\rho) - Q_A^{*} \|_{l^2} \leq \int_0^{\rho} \| Q^{\prime}(\rho^{\prime}) \|_{l^2} d\rho^{\prime} \leq 16J \rho,
\end{equation*}
and \eqref{anti_continuum} follows from the triangle inequality with \eqref{residual}.

\end{proof}

With \Cref{ift}, we turn to orbital stability of the onsite solutions with constants uniform in $\rho$ in the regime $\rho \lesssim J^{-1}$. By verifying the Grillakis–Shatah–Strauss (GSS) criteria uniformly, the dynamics of the constructed sequence \eqref{dnls2} remains trapped within any $O(\epsilon^{\prime})$-tube around the exact ground state, making the explicit sequence and its dynamics a stable and reliable surrogate for the true evolution; see \cite{grillakis1987stability,kapitula2013spectral}.

\begin{corollary}
Assume the hypotheses of \Cref{ift}, and let $Q(\rho)$ be the $C^1$-curve of onsite stationary solutions bifurcating from $Q_{A}^{*}$ at $\rho = 0$. Let $u(t)$ be the evolution with initial datum $u(0) = Q^{(\rho)}$. For any $\epsilon^{\prime}>0$, there exists $c_0(\epsilon^{\prime}) > 0$ such that if $0 \leq \rho < \frac{c_0}{J}$ (by shrinking $\rho$ if necessary), then
\begin{equation}
    \sup_{t \in \mathbb{R}} \inf_{\theta \in \mathbb{R}}\| u(t) - e^{i\theta} Q(\rho) \|_{l^2} < \epsilon^{\prime}.
\end{equation}

\end{corollary}

\begin{proof}

In the rotating frame of reference, linearizing about $Q(\rho)$ with $\eta(t) = a(t) + i b(t)$ where $a(t), b(t) \in l^2(\mathbb{Z};\mathbb{R})$ and $|a|, |b| \ll 1$, the first order approximation yields
\begin{equation*}
    \partial_t
    \begin{pmatrix}
        a \\ b
    \end{pmatrix}
    = J
    \begin{pmatrix}
        L_{+} & 0\\
        0 & L_{-}
    \end{pmatrix}
    \begin{pmatrix}
        a \\ b
    \end{pmatrix},
\end{equation*}
where $J = \begin{pmatrix}
        0 & -I\\
        I & 0
    \end{pmatrix}$
and $L_{\pm}(\rho) = I + \rho \mathscr{L} - c_{\pm} Q(\rho)^2$ with $c_{+} = 3,\ c_{-} = 1$. 

The linearized operators $L_{\pm}(\rho)$ are compact perturbations of $I + \rho \mathscr{L}$ since $Q(\rho)_n \xrightarrow[|n| \rightarrow \infty]{} 0$. By Weyl's theorem on essential spectra, we have a uniform gap $\sigma_{\text{ess}}(L_{\pm}) = \sigma_{\text{ess}}(I + \rho \mathscr{L}) \subseteq \sigma(I + \rho \mathscr{L}) \subseteq [1,\infty)$ since $\mathscr{L}$ is a nonnegative operator. By gauge invariance and IFT, we have $\text{ker}(L_{-}(\rho)) = \text{span}\{Q(\rho)\}$. At $\rho = 0$, both $L_{\pm}(0)$ act as the identity on the orthogonal complement of $\text{span}\{Q_{A}^{*}\}$, and $L_{+}(0)$ has one negative eigenvalue at $-2$, i.e, $n_{-}(L_{+}(0)) = 1$. By Kato’s theorem on stability of the discrete spectrum under bounded perturbations, there exists $c > 0$, independent of $\rho$, such that $\langle L_{\pm}(\rho) v, v \rangle \geq c \| v \|_{l^2}^2$ for any $v \in l^2(\mathbb{Z};\mathbb{R})$ in the orthogonal complement of $Q(\rho)$ (coercivity off the symmetry) and $n_{-}(L_{+}(\rho)) = 1$.

We show the Vakhitov-Kolokolov slope condition is satisfied uniformly. Let $Q_{\omega} = \sqrt{\omega} Q(\rho)$. By chain rule, $N[Q_\omega] = \omega \| Q(\rho) \|_{l^2}^2$, and
\begin{equation*}
    \frac{d}{d\omega} N[Q_\omega] = \| Q(\rho) \|_{l^2}^2 - 2 \rho \langle Q(\rho), \partial_\rho Q(\rho) \rangle = \| Q(\rho) \|_{l^2}^2 + 2 \rho \langle Q(\rho), L_{+}(\rho)^{-1} \mathscr{L} Q(\rho) \rangle,
\end{equation*}
where the second equality is by differentiating $F(Q(\rho),\rho)=0$ by $\rho$ (see \eqref{scaled_gs}). By \eqref{anti_continuum} and \eqref{neumann},
\begin{equation*}
\begin{split}
\frac{d}{d\omega} N[Q_\omega] \geq (\| Q(0) \|_{l^2} - \| Q(\rho) - Q(0) \|_{l^2})^2 - 16J \rho \| Q(\rho) \|_{l^2}^2 > \frac{1}{2},
\end{split}
\end{equation*}
for $\rho \ll J^{-1}$. All GSS criteria have been shown whose constants do not depend on $\rho$. For any $\epsilon^{\prime} > 0$, there exists $\delta > 0$ depending only on $\epsilon^{\prime}$ such that any initial data inside the $\delta$-tube of $Q(\rho)$ stay within the $\epsilon^{\prime}$-tube. By \eqref{anti_continuum}, the proof is complete by choosing $\rho < \frac{\delta}{CJ}$.

\end{proof}

\begin{remark}\label{rmk}
    The first order approximation \eqref{anti_continuum} is not uniform in $\alpha$, since $J^{-1} \sim \alpha$ as $\alpha \rightarrow 0$, and our hypothesis in \Cref{ift} is $\rho \lesssim J^{-1}$. Indeed, the fixed-power definition of PNB does not apply when $\alpha \ll \rho$. For fixed $\epsilon, \omega_A, \omega_B>0$, note that    
    \begin{equation*}
        N_A \sim q_0^2 \sim \frac{2\epsilon}{\alpha},\ N_B \sim 2 g_0^2 \sim \frac{4\epsilon}{\alpha}, \text{ as } \alpha \rightarrow 0. 
    \end{equation*}
   Since $N_B \sim 2 N_A$, the conservation of particle number does not hold. This suggests that the sequences \eqref{dnls2}, \eqref{fdnls5} may show non-negligible deviation away from the exact branches when $\alpha \ll 1$. 
\end{remark}

\subsection{Peierls-Nabarro barrier for fDNLS}

This subsection consists of numerical simulations where we assume $J_n = n^{-(1+\alpha)}$ in which traveling wave solutions are numerically investigated. \Cref{fig:PNB_EPSILON,fig:PNB_OMEGA} most notably highlight that the PNB can vanish under a particular parameter regime and \Cref{fig:vanishingPNB_travelingWave,fig:mobility_figure,fig:6} numerically show traveling wave dynamics. The PNB for fDNLS is fundamentally different from that of the nearest-neighbor interaction. The PNB can take both positive and negative values, and it can increase in $\epsilon$ or $\omega$ depending on $\alpha$. This suggests that the onsite solution need not always be the energy minimizer of the Hamiltonian away from the AC regime. It is of interest to further investigate the role of nonlocality in the framework of variational approach to fDNLS and the stability properties of onsite/offsite solutions. By \eqref{energy} and the symmetry properties of the onsite/offsite solutions, an analogue of \Cref{energy_dnls_onoff} is derived. The proof follows from \Cref{cor} and standard algebra, and thus is omitted.

\begin{proposition}\label{PNB_fdNLS}
Let $q_n,g_n$ be given by \eqref{dnls2}, \eqref{fdnls5}, respectively. Then,
\begin{equation}\label{energy_fdnls}
\begin{split}
E_A &= \epsilon\sum_{n=1}^\infty\Biggl( \frac{|q_n - q_0|^2}{n^{1+\alpha}}+ \frac{1}{2}  \sum_{\substack{m=1 \\ m \neq n}}^{\infty}  \left(\frac{1}{|n-m|^{1+\alpha}} + \frac{1}{|n+m|^{1+\alpha}}\right)|q_n - q_m|^2\Biggl)-\left(\frac{1}{4} q_0^4 + \frac{1}{2} \sum_{n=1}^\infty q_n^4\right),\\
E_B &= \epsilon \sum_{n=2}^\infty \Biggl(\left(\frac{1}{n^{1+\alpha}}+\frac{1}{(n-1)^{1+\alpha}}\right)|g_n - g_0|^2\\
&\hspace{120pt}+ \frac{1}{2}\sum_{\substack{m=2 \\ m\neq n}}^{\infty}\left(\frac{1}{|n-m|^{1+\alpha}}+\frac{1}{|n+m-1|^{1+\alpha}}\right)|g_n - g_m|^2\Biggl) -\frac{1}{2} \sum_{n=1}^\infty g_n^4.
\end{split}
\end{equation}
\end{proposition}

\begin{figure}[H]
\centering
\begin{subfigure}[h]{0.5\linewidth}
\centering
\includegraphics[width=\linewidth]{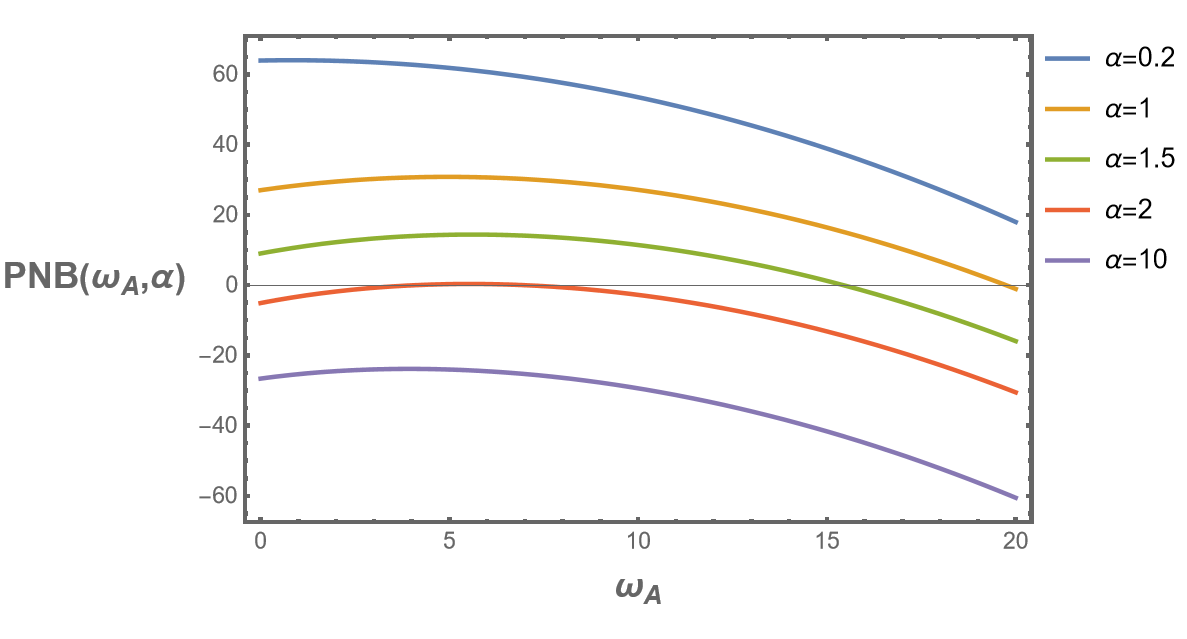}
\caption{}
\label{fig:PNB_OMEGA}
\end{subfigure}%
\begin{subfigure}[h]{0.5\linewidth}
\centering
\includegraphics[width=\linewidth]{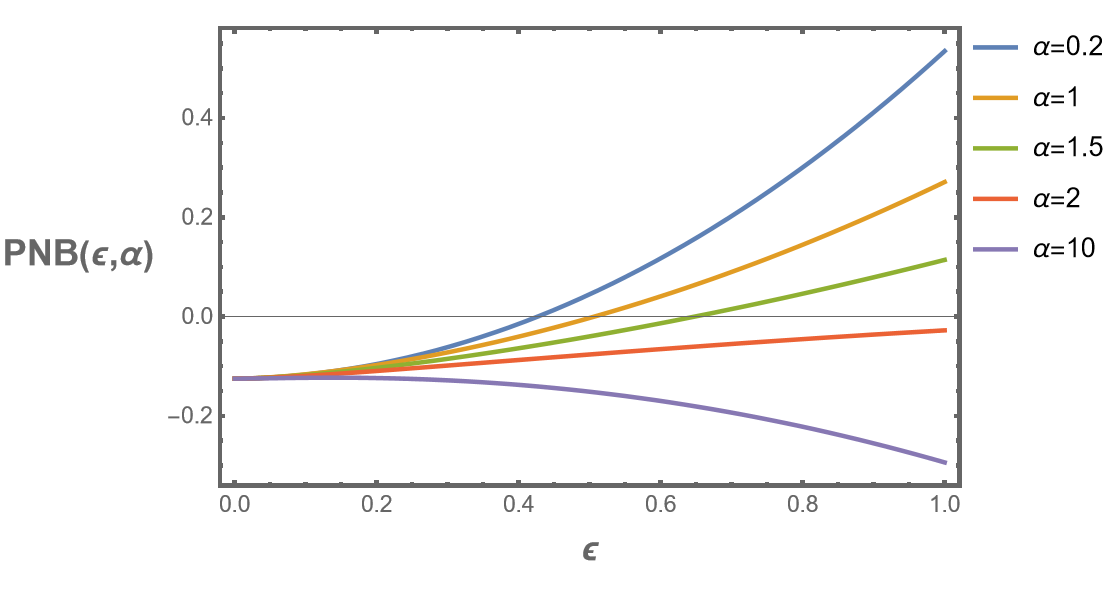}
\caption{ }
\label{fig:PNB_EPSILON}
\end{subfigure}
\label{fig_PNBfDNLS}
\caption{\Cref{fig:PNB_OMEGA} with $\epsilon = 10$ shows the transient state of PNB($\omega$) before they behave as $-\frac{\omega_A^2}{8}$ as $\omega_A \rightarrow \infty$. \Cref{fig:PNB_EPSILON} with $\omega_{A}=1$ shows the PNB diverging away from $-\frac{\omega_A^2}{8}$, the initial value at $\epsilon = 0$, as $\epsilon>0$ increases. The plots show two instances of vanishing PNB. Numerically for $\omega_{A}=1$, the PNB vanishes for $\alpha=0.2,1$ and $\epsilon\approx0.429,\epsilon\approx 0.5066$ respectively. Likewise for \Cref{fig:PNB_OMEGA},  the PNB vanishes for $\alpha = 2$ and $\omega_{A} \in (4.1,4.2)$.}
\end{figure}

The analysis in \Cref{fDNLS} was simplified under the assumption $\rho = \frac{\epsilon}{\omega} 
\ll 1$. A similar analysis is further developed in the context of PNB using \Cref{PNB_fdNLS}. By the conservation of particle number, assume $\omega_A = 2\omega_B$.
\begin{corollary}
As $\omega_A, \omega_B \rightarrow \infty$,
\begin{equation*}\label{energy_fdnls2}
\begin{split}
    E_A &\sim -\frac{\omega_A^2}{4} -(2\zeta(2+2\alpha) + \zeta(1+\alpha)^2)\epsilon^2,\\
    E_B &\sim  -\frac{\omega_B^2}{2}  -  \left(2 \sum_{n=2}^\infty \left(\frac{1}{n^{1+\alpha}} + \frac{1}{(n-1)^{1+\alpha}}\right)^2 + \frac{1}{2}(2\zeta(1+\alpha) - 1)^2\right)\epsilon^2.
\end{split}
\end{equation*}
Therefore, $E_A(\omega_A) \sim 2 E_B(\omega_B)$ and
\begin{equation*}\label{PNB_fdnls_wlarge}
\Delta E_{AB} \sim -\frac{\omega_A^2}{8}.
\end{equation*}
\end{corollary}
\begin{corollary}
    As $\epsilon \rightarrow 0$,
\begin{equation*}\label{energy_fdnls3}
\begin{split}
    E_A &\sim -\frac{\omega_A^2}{4} + \left(-2\zeta(2+2\alpha) + \zeta(1+\alpha)^2\right)\epsilon^2,\\
    E_B &\sim  -\frac{\omega_B^2}{2} + \left(-2\sum_{n=2}^\infty \left(\frac{1}{n^{1+\alpha}} + \frac{1}{(n-1)^{1+\alpha}}\right)^2 + \frac{1}{2}(2\zeta(1+\alpha)-1)^2\right)\epsilon^2.
\end{split}
\end{equation*}
Therefore, $E_A(\omega_A) \sim 2 E_B(\omega_B)$ and
\begin{equation*}\label{PNB_fdnls_esmall}
\Delta E_{AB} \sim -\frac{\omega_A^2}{8} + \left(2\sum_{n=2}^\infty \left(\frac{1}{n^{1+\alpha}} + \frac{1}{(n-1)^{1+\alpha}}\right)^2 - (\zeta(1+\alpha)-1)^2 - 2 \zeta(2+2\alpha) + \frac{1}{2}\right)\epsilon^2.
\end{equation*}
\end{corollary}
For $\omega_A>0$ not sufficiently large, $PNB(\omega_A)$ may not be well approximated by the quadratic term; in fact, $PNB(\omega_A)$ may be increasing for $\omega_A>0$ sufficiently small (see \Cref{fig:PNB_OMEGA}). For $\epsilon>0$ not sufficiently small, the behavior of PNB is generally not quadratic in $\epsilon$ since the higher order terms cannot be neglected. For $\epsilon \ll 1$, note that PNB may increase or decrease depending on the sign of the coefficient of $\epsilon^2$ (see \Cref{fig:PNB_EPSILON}).

\begin{figure}[H]
\begin{subfigure}[h]{0.5\linewidth}
\includegraphics[width=\linewidth]{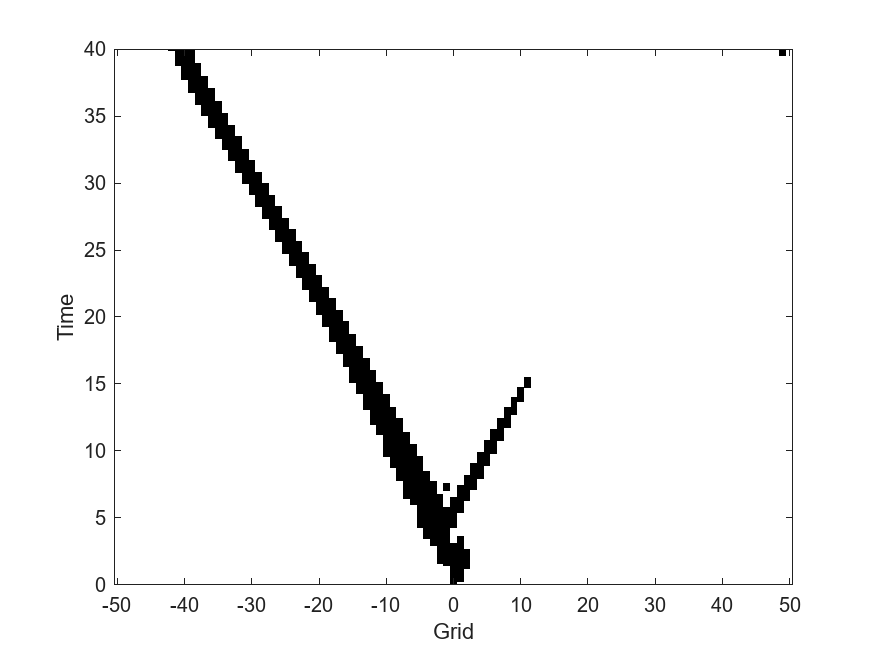}
\caption{Binarized intensity plot. Shaded lattice sites satisfy $|u_n|^2 > 0.14$}
\label{fig:traveling_v1}
\end{subfigure}
\hfill
\begin{subfigure}[h]{0.5\linewidth}
\includegraphics[width=\linewidth]{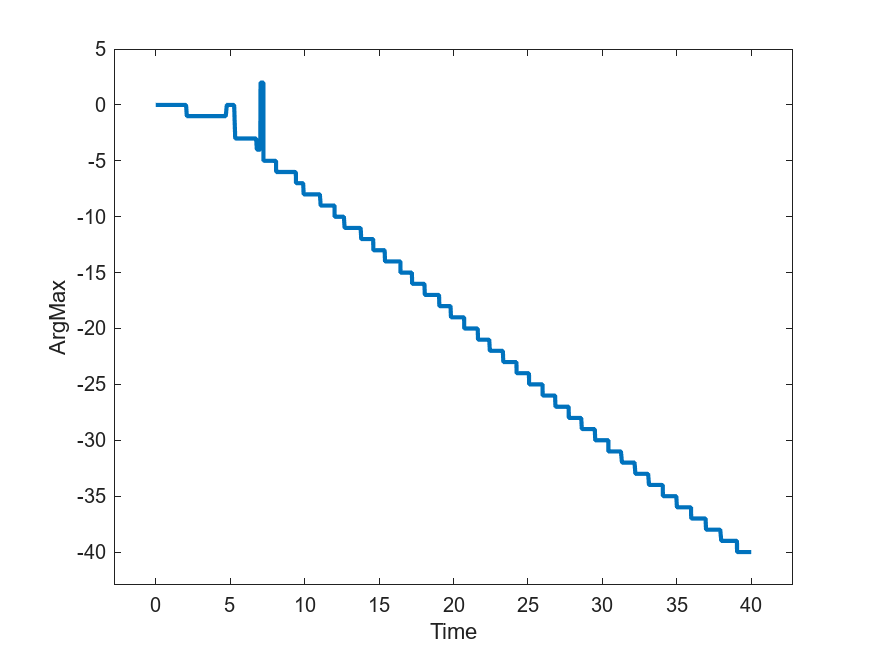}
\caption{ Maximum intensity plot}
\label{fig:staircase_travelingv1}
\end{subfigure}
\caption{Evolution of \cref{generalmodel} with $J_{n}=|n|^{-(1+\alpha)}$, using periodic boundary conditions, and parameters $\alpha=1$, $\omega_{A}=1$, and $\epsilon=0.5066$ with \cref{dnls2} as the initial condition with Galilean drift factor of $e^{i v n}$, where $v=1$. Lattice sites which satisfy the threshold condition $|u_{n}|^{2} > 0.14$ are shaded black in \cref{fig:traveling_v1} whereas \cref{fig:staircase_travelingv1} shows the lattice site which as the maximum intensity per time increment of the evolution.}
\label{fig:vanishingPNB_travelingWave}
\end{figure}

\begin{figure}[H]
\centering
\includegraphics[width = 0.9\textwidth]{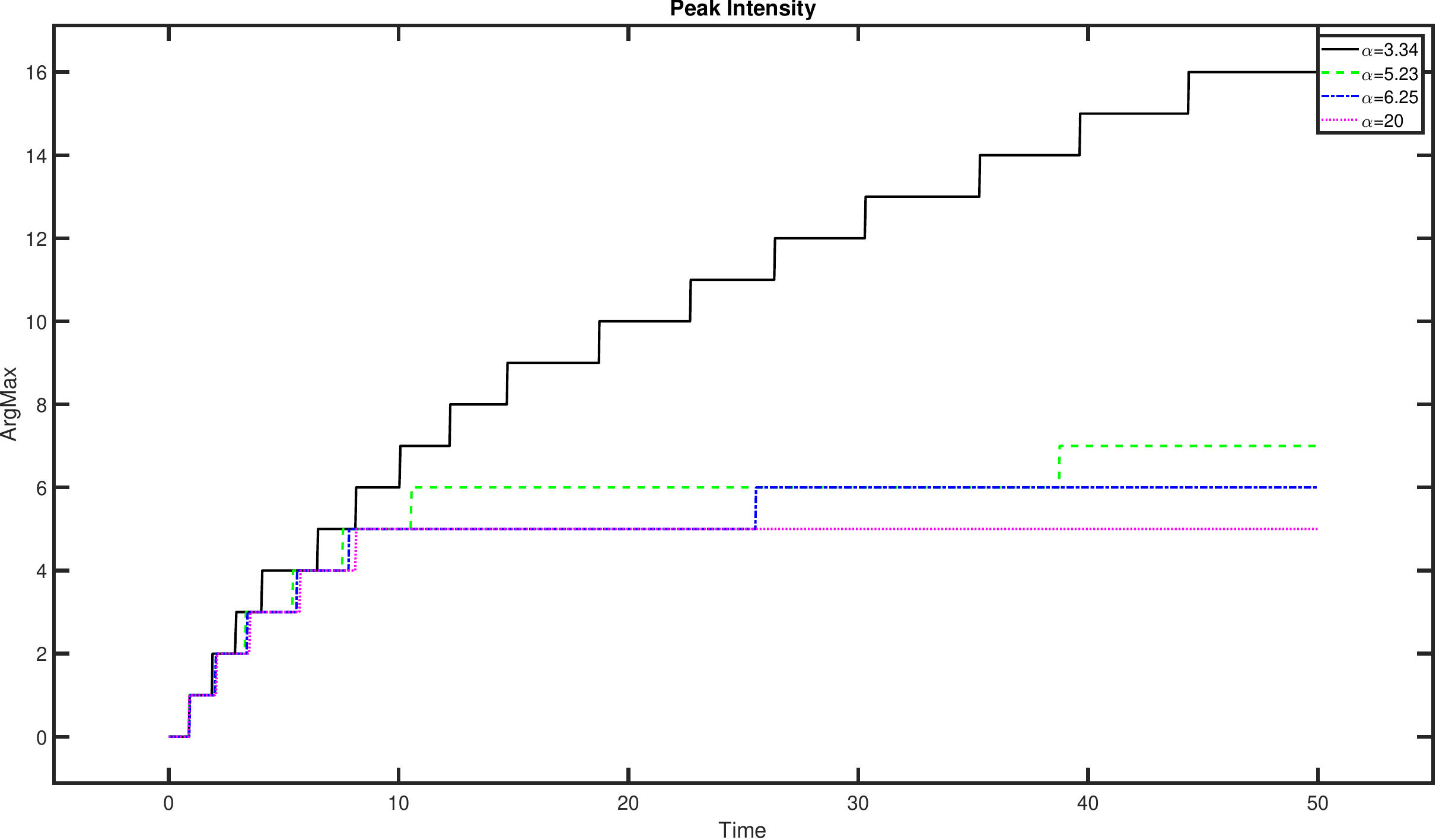}
\caption{Position of peak intensity for traveling waves of the fDNLS for $\epsilon=1, \omega=1$, $v=1$, varying $\alpha \in \{3.34, 5.23,6.25,20\}$, and periodic boundary conditions.}
\label{fig:mobility_figure}
\end{figure}

As pointed out in \Cref{fig:PNB_EPSILON}, there exists a regime under which the PNB is an increasing function in stark contrast to that of DNLS. For $(\omega_A,\epsilon,\alpha)= (1, 0.5066,1)$, we have PNB $\approx 0$ by \Cref{PNB_fdNLS}, which suggests mobility and the existence of a traveling wave solution (TWS). A numerical simulation of the evolution of onsite initial condition \eqref{dnls2} is shown in \Cref{fig:vanishingPNB_travelingWave}. Given the approximations, it is not surprising there is not clear TWS at the start and an dispersive emission at $t \approx 5$. What is surprising is that the dynamics settles into a ``clean" TWS although we believe TWS are non-generic although this demands further research not considered here.

In \Cref{fig:mobility_figure,fig:6}, the mobility/pinning of peak intensity is observed in the nonlocal dynamics given by $J_{n} = |n|^{-(1+\alpha)}$ for $n \neq 0$ with the initial condition as the onsite sequence defined in \eqref{dnls2}. As $\alpha \rightarrow 0$, the nonlocal coupling blows up as $\lim\limits_{\alpha \rightarrow 0+}\zeta(1+\alpha)=\infty$. Moreover the conservation of particle numbers between the onsite and offsite solutions fails in the sense described in \Cref{rmk}, resulting in the erratic behavior illustrated in the top-left plot of \Cref{fig:6}. For $\alpha$ not too small, the intensity drifts and eventually pins at a lattice point. More precisely as $\alpha \to \infty$, the non-locality weakens, leading to a weaker drift (pinning) at earlier times. See \cite[Figure 1]{jenkinson_weinstein_2017} for the drift and pinning under the DNLS dynamics for varying mesh grid sizes $h>0$ whereas \Cref{fig:mobility_figure} plots the argmax of peak intensity for various $\alpha$.

\begin{figure}[H]
\centering
\begin{subfigure}[h]{0.45\linewidth}
\centering
\includegraphics[width=\linewidth]{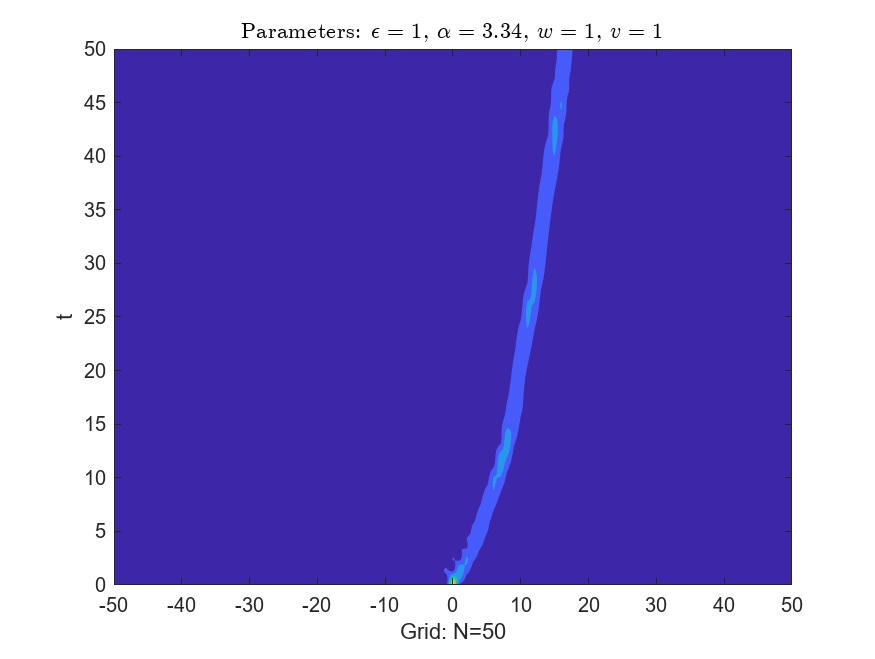}
\caption{}
\label{fig:varyingAlpha_A}
\end{subfigure}
\begin{subfigure}[h]{0.450\linewidth}
\centering
\includegraphics[width=\linewidth]{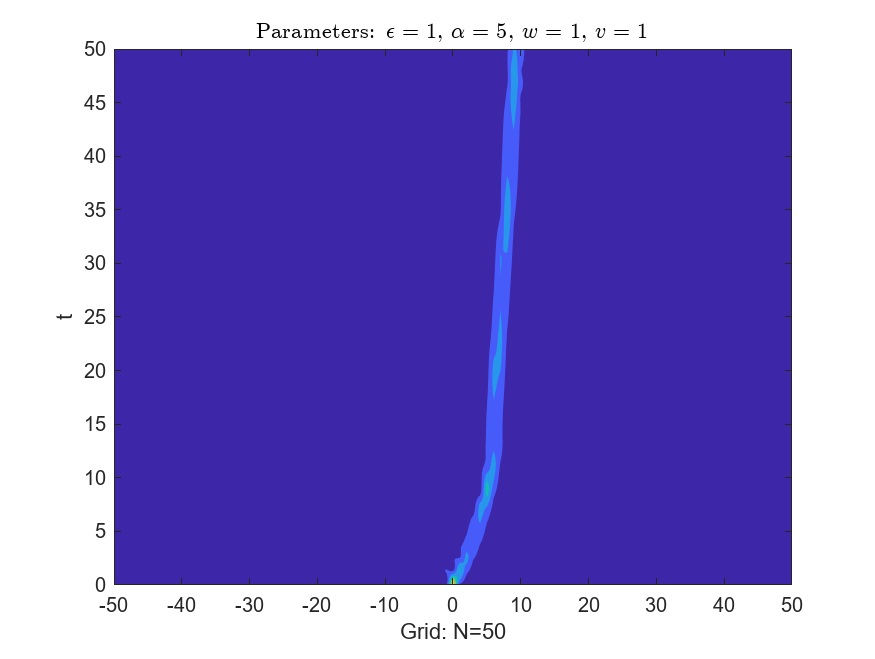}
\caption{}
\label{fig:varyingAlpha_B}
\end{subfigure}
\begin{subfigure}[h]{0.450\linewidth}
\centering
\includegraphics[width=\linewidth]{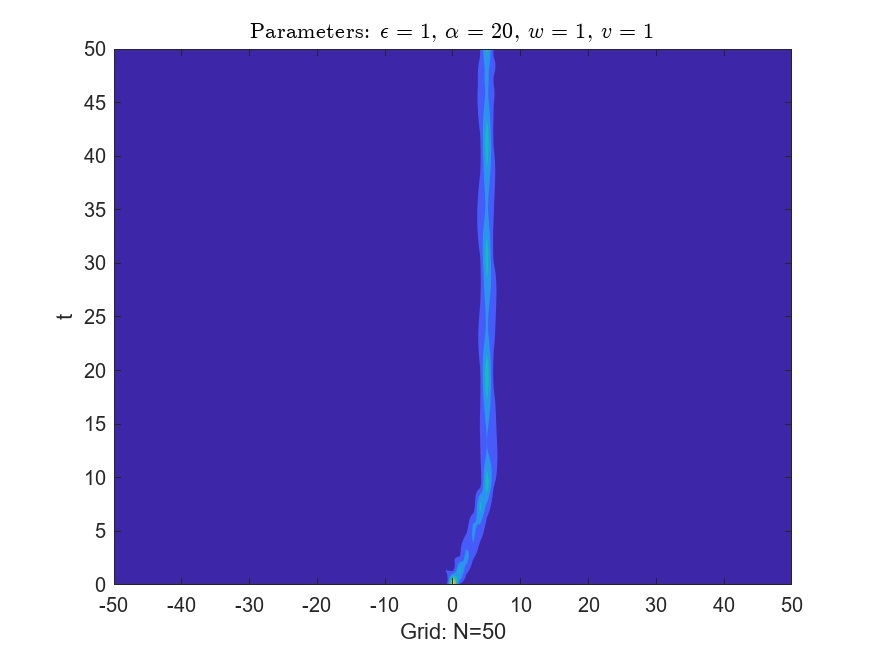}
\caption{}
\label{fig:varyingAlpha_C}
\end{subfigure}
\caption{Intensity plots for varying $\alpha \in \{3.34,5,20\}$ corresponding to \Cref{fig:varyingAlpha_A,fig:varyingAlpha_B,fig:varyingAlpha_C} respectively. Simulations use periodic boundary conditions. }
\label{fig:6}
\end{figure}

\section{Conclusion}

Coherent structures in the long-range fDNLS were quantified: a finite-dimensional spatial-dynamics reduction produced on/offsite stationary profiles with asymptotic validity and proximity to the anti-continuum branches. Modulational instability thresholds were characterized, and simulations showed mobility and pinning across parameters. The Peierls–Nabarro barrier was computed and regimes of near-vanishing PNB were identified. Dynamically, small-data scattering was obtained, and fDNLS was shown to exhibit distinct global behavior for any large $\alpha$. The 
$\alpha$-dependence suggests mechanisms to control transport in long-range lattices. Future work will explore interaction properties (collisions) between localized states and their extensions to 2D lattices with similar coupling functionality.

\section{Acknowledgements}
Both B.C. and A.M. acknowledge support from the NSF RTG award, DMS-1840260.

\appendix
\section{Appendix}\label{appendix}

\begin{proof}[Proof of \Cref{dnls_asymptotic}]
By symmetry, let $n \geq 0$. An explicit computation yields,
\begin{equation*}
    \frac{q_{n+1}}{q_n} = \frac{\rho}{1+2\rho} \xrightarrow[\rho \rightarrow 0]{} 0,
\end{equation*}
and
\begin{equation}\label{dnls_limit}
    \frac{\omega q_n}{\epsilon(q_{n+1}+q_{n-1}-2q_{n}) + q_n^3}=
    \begin{cases}
			\frac{1}{1+\rho^2 (1+2\rho)^{-1} + \rho^{2n}(1+2\rho)^{1-2n}},\ n > 0,\\
            \frac{1}{1+2\rho^2(1+2\rho)^{-1}},\ n=0.
		 \end{cases}
\end{equation}
Taking the limit of the RHS of \eqref{dnls_limit} as $\rho \rightarrow 0$, \eqref{dnls_asymp} is shown.
\end{proof} 

For our numerical simulations, a particular interaction kernel $J_n = |n|^{-(1+\alpha)}$ was used. Consider $(u_n)$ for $n = -N, \dots, N$. Then for the Dirichlet boundary condition, we have
\begin{equation}\label{fraclap_dirichet}
    \mathscr{L}_\alpha u_n = \sum_{-N \leq m \leq N, m \neq n} \frac{u_n - u_m}{|n-m|^{1+\alpha}},
\end{equation}
where $u_m = 0$ for all $|m| > N$.

For the periodic boundary condition, consider the modular arithmetics where the quotient space of $\mathbb{Z}$, with $I_N := \{-N,\dots,N-1\}$ as the fundamental cell, is considered, as $u_{-N} = u_{N}$. Given $m \in \mathbb{Z}$, let $m = 2Nq + r$ where $q \in \mathbb{Z},\ r \in I_N$, and assume $u_m = u_r$. Then for $n \in I_N$,
\begin{equation}\label{fraclap_periodic}
\begin{split}
\mathscr{L}_\alpha u_n &= \sum_{m \neq n} \frac{u_n - u_m}{|n-m|^{1+\alpha}}\\
    &= 2\zeta(1+\alpha)\left(1-\frac{1}{(2N)^{1+\alpha}}\right)u_n - \sum \limits_{r\neq n} c_{nr}(N,\alpha)u_r,    
\end{split}
\end{equation}
where
\begin{equation}\label{fraclap_periodic2}
    c_{nr} := \frac{1}{|n-r|^{1+\alpha}} + \frac{1}{(2N)^{1+\alpha}} \bigg\{ \zeta(1+\alpha, \frac{r-n}{2N}) +
    \zeta(1+\alpha, -\frac{r-n}{2N}) - |\frac{r-n}{2N}|^{-(1+\alpha)}(1+e^{-i(1+\alpha)\pi}) \bigg\},
\end{equation}
and $\zeta(s) = \sum\limits_{k=1}^\infty \frac{1}{k^s}$ is the Riemann zeta function and $\zeta(s,a) = \sum\limits_{k=0}^\infty \frac{1}{(k+a)^s}$ is the Hurwitz zeta function. Lastly we provide a brief derivation of \eqref{fraclap_periodic2}. Using $m = 2Nq + r$ as above,
\begin{equation*}
\begin{split}
\sum_{m \neq n} \frac{u_n - u_m}{|n-m|^{1+\alpha}} &= 2 \zeta(1+\alpha) u_n - \sum_{m \neq n} \frac{u_m}{|n-m|^{1+\alpha}}\\
&= 2 \zeta(1+\alpha) u_n - \sum_{r \in I_N \setminus \{n\}} \frac{u_r}{|n-r|^{1+\alpha}} - \sum_{r \in I_N} \sum_{q \in \mathbb{Z} \setminus \{0\}} \frac{u_r}{|q(2N) + r - n|^{1+\alpha}}    
\end{split}    
\end{equation*}
The last sum simplifies to $2\zeta(1+\alpha) (2N)^{-(1+\alpha)}u_n$ when $r = n$. When $r \neq n$,
\begin{equation*}
\begin{split}
    \sum_{q \in \mathbb{Z} \setminus \{0\}} \frac{(2N)^{1+\alpha}u_r}{|q(2N) + r - n|^{1+\alpha}} &= \sum_{q=1}^\infty\frac{1}{(q+\frac{r-n}{N})^{1+\alpha}} + \sum_{q=1}^\infty\frac{1}{(q-\frac{r-n}{N})^{1+\alpha}}\\
    &=\zeta(1+\alpha,\frac{r-n}{N}) + \zeta(1+\alpha,-\frac{r-n}{N}) - |\frac{r-n}{N}|^{-(1+\alpha)}\left(1+e^{-i(1+\alpha)\pi}\right).
\end{split}
\end{equation*}
Rearranging terms, the expression for $\mathscr{L}_\alpha u_n$ is derived as a matrix multiplication with dense entries.


\bibliographystyle{ieeetr}
\bibliography{nonlocal_references}

\begin{thebibliography}{10}

\bibitem{eisenberg1998discrete}
H.~Eisenberg, Y.~Silberberg, R.~Morandotti, A.~Boyd, and J.~Aitchison, ``Discrete spatial optical solitons in waveguide arrays,'' {\em Physical Review Letters}, vol.~81, no.~16, p.~3383, 1998.

\bibitem{aceves1996discrete}
A.~Aceves, C.~De~Angelis, T.~Peschel, R.~Muschall, F.~Lederer, S.~Trillo, and S.~Wabnitz, ``Discrete self-trapping, soliton interactions, and beam steering in nonlinear waveguide arrays,'' {\em Physical Review E}, vol.~53, no.~1, p.~1172, 1996.

\bibitem{anderson1998macroscopic}
B.~P. Anderson and M.~A. Kasevich, ``Macroscopic quantum interference from atomic tunnel arrays,'' {\em Science}, vol.~282, no.~5394, pp.~1686--1689, 1998.

\bibitem{trombettoni2001discrete}
A.~Trombettoni and A.~Smerzi, ``Discrete solitons and breathers with dilute {B}ose-{E}instein condensates,'' {\em Physical Review Letters}, vol.~86, no.~11, p.~2353, 2001.

\bibitem{Mingaleev1999}
S.~F. Mingaleev, P.~L. Christiansen, Y.~Gaididei, M.~Johansson, and K.~{\O}. Rasmussen, ``Models for {E}nergy and {C}harge {T}ransport and {S}torage in {B}iomolecules,'' {\em Journal of Biological Physics}, vol.~25, pp.~41--63, Mar 1999.

\bibitem{PhysRevE.55.6141}
Y.~B. Gaididei, S.~F. Mingaleev, P.~L. Christiansen, and K.~O. Rasmussen, ``Effects of nonlocal dispersive interactions on self-trapping excitations,'' {\em Phys. Rev. E}, vol.~55, pp.~6141--6150, May 1997.

\bibitem{kevrekidis}
P.~G. Kevrekidis, {\em The discrete nonlinear {S}chr{\"o}dinger equation: mathematical analysis, numerical computations and physical perspectives}, vol.~232.
\newblock Springer Science \& Business Media, 2009.

\bibitem{SulemSulem:NLS_book}
P.-L.~S. Catherine~Sulem, {\em The {N}onlinear {S}chrödinger {E}quation: {S}elf-focusing and {W}ave {C}ollapse}.
\newblock Applied Mathematical Sciences, Springer, 1~ed., 1999.

\bibitem{molina2020two}
M.~I. Molina, ``The two-dimensional fractional discrete nonlinear {S}chr{\"o}dinger equation,'' {\em Physics Letters A}, vol.~384, no.~33, p.~126835, 2020.

\bibitem{johansson1998switching}
M.~Johansson, Y.~B. Gaididei, P.~L. Christiansen, and K.~Rasmussen, ``Switching between bistable states in a discrete nonlinear model with long-range dispersion,'' {\em Physical Review E}, vol.~57, no.~4, p.~4739, 1998.

\bibitem{rasmussen1998localized}
K.~Rasmussen, P.~L. Christiansen, M.~Johansson, Y.~B. Gaididei, and S.~Mingaleev, ``Localized excitations in discrete nonlinear {S}chr{\"o}dinger systems: effects of nonlocal dispersive interactions and noise,'' {\em Physica D: Nonlinear Phenomena}, vol.~113, no.~2-4, pp.~134--151, 1998.

\bibitem{longhi2015fractional}
S.~Longhi, ``Fractional {S}chr{\"o}dinger equation in optics,'' {\em Optics letters}, vol.~40, no.~6, pp.~1117--1120, 2015.

\bibitem{kirkpatrick2013continuum}
K.~Kirkpatrick, E.~Lenzmann, and G.~Staffilani, ``On the continuum limit for discrete {NLS} with long-range lattice interactions,'' {\em Communications in Mathematical Physics}, vol.~317, no.~3, pp.~563--591, 2013.

\bibitem{malomed2024}
M.~Zhong, B.~A. Malomed, and Z.~Yan, ``Dynamics of discrete solitons in the fractional discrete nonlinear {S}chr\"odinger equation with the quasi-{R}iesz derivative,'' {\em Phys. Rev. E}, vol.~110, p.~014215, Jul 2024.

\bibitem{Weinstein1983NonlinearSE}
M.~I. Weinstein, ``Nonlinear {S}chr{\"o}dinger equations and sharp interpolation estimates,'' {\em Communications in Mathematical Physics}, vol.~87, pp.~567--576, 1983.

\bibitem{weinstein1989nonlinear}
M.~I. Weinstein, ``The nonlinear {S}chr{\"o}dinger equation—singularity formation, stability and dispersion,'' {\em The Connection between Infinite Dimensional and Finite Dimensional Dynamical Systems}, pp.~213--232, 1989.

\bibitem{ZAKHAROV2009540}
V.~Zakharov and L.~Ostrovsky, ``Modulation instability: The beginning,'' {\em Physica D: Nonlinear Phenomena}, vol.~238, no.~5, pp.~540--548, 2009.

\bibitem{dauxois1993energy}
T.~Dauxois and M.~Peyrard, ``Energy localization in nonlinear lattices,'' {\em Physical Review Letters}, vol.~70, no.~25, p.~3935, 1993.

\bibitem{daumont1997modulational}
I.~Daumont, T.~Dauxois, and M.~Peyrard, ``Modulational instability: first step towards energy localization in nonlinear lattices,'' {\em Nonlinearity}, vol.~10, no.~3, p.~617, 1997.

\bibitem{gori2013}
G.~Gori, T.~Macr{\`\i}, and A.~Trombettoni, ``Modulational instabilities in lattices with power-law hoppings and interactions,'' {\em Physical Review E—Statistical, Nonlinear, and Soft Matter Physics}, vol.~87, no.~3, p.~032905, 2013.

\bibitem{christiansen1998solitary}
P.~L. Christiansen, Y.~B. Gaididei, M.~Johansson, K.~Rasmussen, V.~Mezentsev, and J.~J. Rasmussen, ``Solitary excitations in discrete two-dimensional nonlinear {S}chr{\"o}dinger models with dispersive dipole-dipole interactions,'' {\em Physical Review B}, vol.~57, no.~18, p.~11303, 1998.

\bibitem{hadvzievski2003}
L.~Had{\v{z}}ievski, M.~Stepi{\'c}, and M.~M. {\v{S}}kori{\'c}, ``Modulation instability in two-dimensional nonlinear schr{\"o}dinger lattice models with dispersion and long-range interactions,'' {\em Physical Review B}, vol.~68, no.~1, p.~014305, 2003.

\bibitem{Copeland:20}
A.~Aceves and A.~Copeland, ``Spatiotemporal {D}ynamics in the {F}ractional {N}onlinear schr\"{o}dinger equation,'' {\em Frontiers, Nonlinear Photonics}, vol.~3, 2022.

\bibitem{DINH2019117}
V.~D. Dinh, ``Blow-up criteria for fractional nonlinear {S}chr{\"o}dinger equations,'' {\em Nonlinear Analysis: Real World Applications}, vol.~48, pp.~117--140, 2019.

\bibitem{1534-0392_2015_6_2265}
Y.~Hong and Y.~Sire, ``On {F}ractional {S}chrödinger {E}quations in {S}obolev spaces,'' {\em Communications on Pure \& Applied Analysis}, vol.~14, no.~6, pp.~2265--2282, 2015.

\bibitem{dinh:hal-01426761}
V.~D. Dinh, ``{Well-posedness of nonlinear fractional {S}chr{\"o}dinger and wave equations in {S}obolev spaces},'' {\em {International Journal of Applied Mathematics}}, vol.~31, pp.~483--525, Sept. 2018.

\bibitem{choi2022well}
B.~Choi and A.~Aceves, ``Well-posedness of the mixed-fractional nonlinear {S}chr{\"o}dinger equation on $\mathbb{R}^2$,'' {\em Partial Differential Equations in Applied Mathematics}, vol.~6, p.~100406, 2022.

\bibitem{tarasov2006continuous}
V.~E. Tarasov, ``Continuous limit of discrete systems with long-range interaction,'' {\em Journal of Physics A: Mathematical and General}, vol.~39, no.~48, p.~14895, 2006.

\bibitem{hong2019strong}
Y.~Hong and C.~Yang, ``Strong convergence for discrete nonlinear {S}chr\"odinger equations in the continuum limit,'' {\em SIAM Journal on Mathematical Analysis}, vol.~51, no.~2, pp.~1297--1320, 2019.

\bibitem{choi2023continuum}
B.~Choi and A.~Aceves, ``Continuum limit of 2{D} fractional nonlinear {S}chr{\"o}dinger equation,'' {\em Journal of Evolution Equations}, vol.~23, no.~2, p.~30, 2023.

\bibitem{JenWeinLocal}
M.~Jenkinson and M.~I. Weinstein, ``Onsite and offsite bound states of the discrete nonlinear {S}chrödinger equation and the {P}eierls–{N}abarro barrier,'' {\em Nonlinearity}, vol.~29, p.~27, dec 2015.

\bibitem{jenkinson_weinstein_2017}
M.~Jenkinson and M.~I. Weinstein, ``Discrete solitary waves in systems with nonlocal interactions and the {P}eierls--{N}abarro barrier,'' {\em Communications in Mathematical Physics}, vol.~351, pp.~45--94, 2017.

\bibitem{weinstein1999excitation}
M.~I. Weinstein, ``Excitation thresholds for nonlinear localized modes on lattices,'' {\em Nonlinearity}, vol.~12, no.~3, p.~673, 1999.

\bibitem{stefanov2005asymptotic}
A.~Stefanov and P.~G. Kevrekidis, ``Asymptotic behaviour of small solutions for the discrete nonlinear {S}chr{\"o}dinger and {K}lein--{G}ordon equations,'' {\em Nonlinearity}, vol.~18, no.~4, p.~1841, 2005.

\bibitem{stein1993harmonic}
E.~M. Stein and T.~S. Murphy, {\em Harmonic analysis: real-variable methods, orthogonality, and oscillatory integrals}, vol.~3.
\newblock Princeton University Press, 1993.

\bibitem{keel1998endpoint}
M.~Keel and T.~Tao, ``Endpoint {S}trichartz estimates,'' {\em American Journal of Mathematics}, vol.~120, no.~5, pp.~955--980, 1998.

\bibitem{mackay1994proof}
R.~MacKay and S.~Aubry, ``Proof of existence of breathers for time-reversible or hamiltonian networks of weakly coupled oscillators,'' {\em Nonlinearity}, vol.~7, no.~6, p.~1623, 1994.

\bibitem{cuevas2008approximation}
J.~Cuevas, G.~James, P.~G. Kevrekidis, B.~A. Malomed, and B.~Sanchez-Rey, ``Approximation of solitons in the discrete {NLS} equation,'' {\em Journal of Nonlinear Mathematical Physics}, vol.~15, no.~Suppl 3, pp.~124--136, 2008.

\bibitem{KivCam}
Y.~S. Kivshar and D.~K. Campbell, ``Peierls-{N}abarro potential barrier for highly localized nonlinear modes,'' {\em Phys. Rev. E}, vol.~48, pp.~3077--3081, Oct 1993.

\bibitem{flach1998breathers}
S.~Flach, ``Breathers on lattices with long range interaction,'' {\em {P}hysical {R}eview {E}}, vol.~58, no.~4, p.~R4116, 1998.

\bibitem{grillakis1987stability}
M.~Grillakis, J.~Shatah, and W.~Strauss, ``Stability theory of solitary waves in the presence of symmetry, {I},'' {\em Journal of functional analysis}, vol.~74, no.~1, pp.~160--197, 1987.

\bibitem{kapitula2013spectral}
T.~Kapitula, K.~Promislow, {\em et~al.}, {\em Spectral and dynamical stability of nonlinear waves}, vol.~457.
\newblock Springer, 2013.

\end{thebibliography}

\end{document}